\documentclass[a4paper,11pt]{amsart}

\usepackage{color}
\usepackage{comment}
\usepackage{amssymb}
\usepackage{amsfonts}

\usepackage{amscd}

\usepackage{slashed}
\usepackage{pdflscape}
\usepackage{tikz}
\usepackage{enumerate}
\usepackage{multirow}
\usepackage{stmaryrd}

\usepackage{bbold}
\usepackage[backref=page,pagebackref,linkcolor=blue]{hyperref}
\usepackage{mathrsfs}
\usepackage{bbding}

\newcommand{\bm}[1]{\mbox{\boldmath $ #1 $}}

\newtheorem{theorem}{Theorem}[section]
\newtheorem{lemma}[theorem]{Lemma}
  
\newtheorem{proposition}[theorem]{Proposition}
\newtheorem{corollary}[theorem]{Corollary}

\theoremstyle{definition}
\newtheorem{definition}[theorem]{Definition}
\newtheorem{example}[theorem]{Example}

\theoremstyle{remark}
\newtheorem{remark}[theorem]{Remark}

\newtheorem{problem}[theorem]{Problem}

\usepackage{pdfsync}

\usepackage{graphicx} 
\usepackage{amsmath} 
\usepackage{amsfonts}
\usepackage{amssymb}

\newcommand{\be}{\begin{equation}}

\newcommand{\ee}{\end{equation}}

\newcommand{\g}{g^o}

\newcommand{\II}{{\bf\rm I\hspace{-.2mm}I}}
\newcommand{\IIo}{\mathring{{\bf\rm I\hspace{-.2mm} I}}{\hspace{.2mm}}}

\newcommand{\si}{\sigma}

\newcommand{\ba}{\begin{array}}

\newcommand{\ea}{\end{array}}

\newcommand{\beq}{\begin{eqnarray}}

\newcommand{\eeq}{\end{eqnarray}}

\newtheorem{lm}{lemma}

\newtheorem{thee}{theorem}

\newtheorem{proo}{proposition}

\newtheorem{co}{corollary}

\newtheorem{rem}{remark}

\newtheorem{deff}{definition}

\newcommand{\bd}{\begin{deff}}

\newcommand{\ed}{\end{deff}}

\newcommand{\bl}{\begin{lm}}

\newcommand{\el}{\end{lm}}

\newcommand{\bp}{\begin{proo}}

\newcommand{\ep}{\end{proo}}

\newcommand{\bt}{\begin{thee}}

\newcommand{\et}{\end{thee}}

\newcommand{\bc}{\begin{co}}

\newcommand{\ec}{\end{co}}

\newcommand{\brm}{\begin{rem}}

\newcommand{\erm}{\end{rem}}

\newcommand{\F}{\overline{F}}

\usepackage{t1enc}
\def\frak{\mathfrak}

\def\Cal{\mathcal}

\newcommand{\newc}{\newcommand}

\let\ccdot\cdot
\def\cdot{\hbox to 2.5pt{\hss$\ccdot$\hss}}

\newc{\aR}{\mbox{\boldmath{$ R$}}}
\newc{\aS}{\mbox{\boldmath{$ S$}}}
\newc{\aT}{\mbox{\boldmath{$ T$}}}
\newc{\aW}{\mbox{\boldmath{$ W$}}}

\newc{\aD}{\mbox{\boldmath{$ D$}}\hspace{-.2mm}}

\newc{\aK}{\mbox{\boldmath{$ K$}}}
\newc{\aL}{\mbox{\boldmath{$ L$}}}
\newcommand{\ce}{{\Cal E}}

\newcommand{\ct}{{\Cal T}}

\newcommand{\hD}{\widehat{D}}

\usepackage{amssymb}
\usepackage{amscd}

\newcommand{\nd}{\nabla}

\newcommand{\Rho}{{\rm P}}
\newcommand{\Up}{\Upsilon}

\newcommand{\Ric}{\operatorname{Ric}}
\newcommand{\Sc}{\operatorname{Sc}}

\newcommand{\cT}{{\mathcal T}}

\newcommand{\nn}[1]{(\ref{#1})}

\newcommand{\bg}{\mbox{\boldmath{$ g$}}}

\newcommand{\J}{{\mbox{\sf J}}}

\newc{\obstrn}[2]{B^{#1}_{#2}}

\newcommand{\rpl}                         % +) or <+
{\mbox{$
\begin{picture}(12.7,8)(-.5,-1)
\put(0,0.2){$+$}
\put(4.2,2.8){\oval(8,8)[r]}
\end{picture}$}}

\newcommand{\lpl}                         % (+ or +>
{\mbox{$
\begin{picture}(12.7,8)(-.5,-1)
\put(2,0.2){$+$}
\put(6.2,2.8){\oval(8,8)[l]}
\end{picture}$}}

\usepackage{ifthen}

\newc{\tensor}[1]{#1}
\newc{\Mvariable}[1]{\mbox{#1}}
\newc{\down}[1]{{}_{#1}}
\newc{\up}[1]{{}^{#1}}

\newc{\JulyStrut}{\rule{0mm}{6mm}}
\newc{\midtenPan}{\mbox{\sf S}}
\newc{\midten}{\mbox{\sf T}}
\newc{\midtenEi}{\mbox{\sf U}}
\newc{\ATen}{\mbox{\sf E}}
\newc{\BTen}{\mbox{\sf F}}
\newc{\CTen}{\mbox{\sf G}}

\def\sideremark#1{\ifvmode\leavevmode\fi\vadjust{\vbox to0pt{\vss
 \hbox to 0pt{\hskip\hsize\hskip1em
 \vbox{\hsize2cm\tiny\raggedright\pretolerance10000
  \noindent #1\hfill}\hss}\vbox to8pt{\vfil}\vss}}}

\newcommand{\edz}[1]{\sideremark{#1}}

\numberwithin{equation}{section}

\renewcommand\F{{\mathcal F}}

\newcommand{\B}{\mathcal B}

\newcommand{\nablab}{\bar\nabla}
\newcommand{\cc}{\boldsymbol{c}}

\renewcommand{\=}{\stackrel\Sigma =}

\newcommand{\nablat}{\nabla^\top}
\newcommand{\nablan}{\nabla_n}

\DeclareMathOperator{\tr}{tr}

\newcommand{\db}{\scalebox{.93}{$\bar{d}\hspace{.2mm}$}}
\newcommand{\dbs}{\scalebox{.7}{$\bar d^{\phantom{A^a}\!\!\!\!\!} \hspace{.1mm}$}}

\newcommand{\dd}{\boldsymbol{d}}
\newcommand{\Z}{{\mathcal Z}}

\newcommand{\Wn}{\raisebox{.35mm}{$\stackrel{\scriptscriptstyle n}{W}$}{}}
\newcommand{\Whn}{\raisebox{.35mm}{$\stackrel{\scriptscriptstyle \hat n}{W}$}{}}
\newcommand{\Chn}{\raisebox{.35mm}{$\stackrel{\scriptscriptstyle \hat n}{C}$}{}}
\newcommand{\Cn}{\raisebox{.35mm}{$\stackrel{\scriptscriptstyle  n}{C}$}{}}
\newcommand{\Rn}{\raisebox{.35mm}{$\stackrel{\scriptscriptstyle  n}{R}$}{}}

\newcommand{\sss}{\scriptscriptstyle}
\newcommand{\Bach}{{\rm B}}

\newcommand{\Pre}{{\mathscr P}}

\newcommand{\sideline}[1]{#1 |}
\newcommand{\measure}
{\sideline{\hat \delta  s}}
\begin{document}

\renewcommand{\today}{}
\title{
{Variational Calculus for  Hypersurface Functionals:} \\[2mm]
{  Singular Yamabe Problem Willmore Energies}\\[4mm]
}
\author{Michael Glaros${}^\natural$, A. Rod Gover${}^\sharp$, Matthew Halbasch${}^\natural$ \& Andrew Waldron${}^\natural$}

\address{${}^\sharp$Department of Mathematics\\
  The University of Auckland\\
  Private Bag 92019\\
  Auckland 1\\
  New Zealand,  and\\
  Mathematical Sciences Institute, Australian National University, ACT
  0200, Australia} \email{gover@math.auckland.ac.nz}
  
  \address{${}^{\natural}$Department of Mathematics\\
  University of California\\
  Davis, CA95616, USA} \email{glaros,halbasch,wally@math.ucdavis.edu}

\vspace{10pt}

\renewcommand{\arraystretch}{1}

\begin{abstract} 
We develop the  calculus for 
hypersurface variations 
 based on variation of the hypersurface defining function.
  This is used to show that the functional gradient  of a  new Willmore-like, conformal  hypersurface energy agrees exactly with the obstruction to smoothly solving the singular Yamabe problem for conformally compact four-manifolds. We give explicit hypersurface formul\ae\ for both the energy functional  and the obstruction.

\vspace{15cm}

\noindent
%\begin{keywords}

\end{abstract}

%58E30  	Variational principles
%49S05  	Variational principles of physics
%51P05  	Geometry and physics
%53A30  	Conformal differential geometry
%53A55  	Differential invariants (local theory), geometric objects
%53C80  	Applications to physics (Global differential geoemtry)
%53B50  	Applications to physics (Local differential geometry)
%53A07  	Higher-dimensional and -codimensional surfaces in Euclidean $n$-space
%53C21  	Methods of Riemannian geometry, including PDE methods; curvature restrictions
%53B15  	Other connections
%83C99  	None of the above, but in this section (General relativity)

\maketitle

\pagestyle{myheadings} \markboth{Glaros, Gover, Halbasch \& Waldron}{Variational Calculus and the  Singular Yamabe Problem}

\newpage

\tableofcontents

\section{Introduction}

Hypersurface geometry is critically important for the analysis of boundary problems in mathematics and physics. 
Conformal hypersurfaces form 
a subtle but important class of  cases, particularly because of their role in treating the 
boundary at infinity of conformally compact structures. 
Naturally the fundamental problem is the construction and study of local and global invariants. The work~\cite{GW15} develops 
an effective approach to hypersurface geometries by treating   them as conformal infinities (see also the announcement~\cite{CRMouncementCRM}).
This approach includes a calculus for hypersurface invariants based on  defining 
functions. 

In this article we  develop the corresponding  variational theory; namely a hypersurface variational calculus 
 based around defining functions.
A main result is Theorem~\ref{VARY}
which provides the general formula for computing variations of embeddings of
hypersurface functionals 
by varying the underlying hypersurface  defining function. A key motivation 
for our development of that theory was
to develop an effective approach to computing the Euler--Lagrange equations corresponding to the new conformal  
energies of~\cite{GW15}. 
These functionals generalise to higher dimensions the Willmore energy~\cite{Willmore}, also known as the rigid string action~\cite{Polyakov}. For the crucial case of four manifolds, we show that the gradient of this functional agrees with the obstruction to smoothly solving the Yamabe problem for conformally compact structures.

The Willmore energy and its associated gradient provide important
invariants for surfaces. Indeed the Willmore conjecture
concerning absolute minimizers of this energy has stimulated much
geometric analysis~\cite{Riviere}; the general proof of the conjecture was
provided by Marques-Neves~\cite{Marques,Marques1}. In physics, the Willmore energy
functional has applications in string theory where it appears as the entanglement entropy for a three dimensional conformal field theory~\cite{Astaneh}. 
 It also arises as a certain log
coefficient in a holographic notion of physical observables, and is
closely related to the so-called conformal anomaly
\cite{GrahamWitten}.
Furthermore it appears as the renormalized area of a minimal surface embedded in  a hyperbolic three-manifold~\cite{Alexakis}.
An important aspect of these invariants is their conformal
invariance.  While the Willmore energy integrand is not surprising its
gradient, with respect to variations of embeddings, is an important
local conformal invariant because it has linear leading term. 
We call this gradient the {\em Willmore invariant}.

These applications and observations suggest a potentially significant
role in geometry and physics for corresponding conformal invariants in
higher dimensions. For four-dimensional hypersurfaces  in Euclidean spaces,  conformal energy
functionals have been provided in~\cite{Guven} and~\cite{YuriThesis}. Recently it was shown that  for
hypersurfaces of every dimension, there exists a conformally invariant canonical  analog of both the  Willmore energy and the Willmore invariant.   These are
generated and constructed uniformly via a singular Yamabe problem that
places the hypersurface as the boundary at infinity of a conformally
compact, constant scalar curvature, manifold
\cite{CRMouncementCRM,GW15}.
In each  higher dimension, the   analog of the Willmore invariant is called the
{\it obstruction density}.
This work was
 partly inspired by the results~\cite{ACF} of
Andersson--Chru{\'s}ciel--Friedrich who studied boundary
asymptotics for this Loewner--Nirenberg-type problem~\cite{Loewner}. They found that the three-manifold 
obstruction to smoothness was a  conformal invariant and 
also gave some information about the corresponding obstructions in higher dimensions. For three-manifold boundaries the {obstruction density} is  easily
shown to agree with the Willmore invariant.  

In~\cite{CRMouncementCRM,GW15}  the singular Yamabe problem was recast as the tool for
treating hypersurface conformal geometry that is analogous to the
Fefferman-Graham Poincar\'e-metric as an approach to intrinsic conformal
geometry. From this perspective, and using tractor calculus to
re-express the Yamabe equation, the problem of computing the
obstruction is reduced to a simple algorithm that also reveals the
qualitative aspects of the obstruction densities and higher Willmore energies. In particular there is  a dimension
parity dichotomy: For each even hypersurface dimension, the  obstruction density is a  linear-leading-order conformal invariant that is a   close analog of the
Willmore invariant (and should be considered as a fundamental scalar
curvature invariant); via elementary representation theory, odd
dimensional hypersurfaces cannot admit linear leading order
invariants. 

It is an interesting problem to understand  the higher Willmore energies and obstruction densities for odd dimensional hypersurfaces.  
A  particularly important case for both mathematics and physics is  when these are  four-manifold boundaries. In this article we employ the general holographic formula for the obstruction density~\cite{GW15}   to compute the four-manifold obstruction density:
 \begin{proposition}\label{Calzone}
The  
obstruction density for~$\db = 3$ is given by
$$
\B_3 \!= \!\frac16\Big[ \!L^{ab\,}\! \!
%\circ 
\hspace{.2mm}
\big(3\IIo^2_{\!(ab)\!\hspace{.2mm}\circ} \hspace{.4mm}\!\!-\! \hspace{.2mm}\Whn_{ab}\big) -
%\tr(\IIo B) 
\IIo^{ab}\!B_{ab}
 + K^2\! - 7\Whn^{ab}\IIo^2_{ab} + 2\Whn_{ab}\Whn^{ab}\! + \IIo^{ab}\IIo^{cd}W_{\!cabd} + \Whn\!_{abc}^\top\Whn^{abc}\Big] .
$$
\end{proposition}
\noindent
Theorem~\ref{bigformaggio}  
recalls that $\B_{\bar d}$ is a conformally invariant density of weight $-d$, and 
our tensor notations are explained in Section~\ref{conventions} below,
while the hypersurface invariants appearing in the above formula are introduced in Section~\ref{hypersurfaces}. In particular~$\IIo$ denotes the trace-free second fundamental form, while $L^{ab}$ is a
second order, conformally invariant differential operator (that appears  in Bernstein--Gel'fand--Gel'fand~(BGG) sequences) such that $L^{ab}\IIo_{ab}=\nablab^a\nablab^b \IIo_{ab}+\mbox{\it lower order terms}$.

As mentioned above, another main result  in~\cite{GW15} is the construction of
conformally invariant functionals that generalise the Willmore energy to higher dimensions. For a hypersurface $\Sigma$ of dimension $\bar{d}$ these
take the form 
$$
\mathcal{I}_{\bar d}=\alpha_{\bar d} \int_\Sigma dA_{\bar g}\,  N_A {\sf P}_{\bar{d}}\, N^A \, ,
$$ 
where $\alpha_{\bar d}$ is a
non-vanishing constant and 
 ${\sf P}_{\bar{d}}$ is a uniformly constructed family of
conformally invariant operators, and $N^A$ is the normal tractor of~\cite{BEG}  introduced in Section~\ref{htractors}. For $\bar{d}$ even, the operators ${\sf P}_{\bar d}$ take the form
$(\Delta^\top)^{\frac{\bar{d}}{2}}+\mbox{\it lower order terms}$,
so are hypersurface analogs of the GJMS operators of Graham, Jennes, Mason and Sparling~\cite{GJMS}. For $\db$ both even and odd 
a holographic formula for ${\sf P}_{\bar d}$ is known. This
 gives that  ${\sf P}_3$ takes the form $\IIo^{ab}\nabla^\top_a \nabla_b^\top+\mbox{\it lower order terms}$
(the detailed formula appears in Equation~\nn{P23}). So
 for generic hypersurfaces ${\sf P}_3$ is Laplacian-like. Using these results for ${\sf P}_3$ yields 
 \begin{equation}\label{I3}
\mathcal{I}_3= \frac16 \int_\Sigma\, dA_{\bar g} \, \IIo^{ab}\mathcal{F}_{ab}\, ,  
\end{equation}
where the conformally invariant Fialkow tensor takes the form ${\mathcal F}_{ab}=\IIo^2_{ab}$ $+$ $\mbox{\it lower}$ $\mbox{\it order terms}$; see Equation~\nn{Fialkow} and $\alpha_{3}=\frac16$ has been chosen for later convenience.

Our final main result is that the functional gradient of the energy~${\mathcal I}_3$ agrees precisely with the obstruction ${\mathcal B}_3$ to the singular Yamabe problem mentioned above:

\begin{proposition}\label{gradient}
 With respect to hypersurface variations, the  gradient of ${\mathcal I}_3$ is ${\mathcal B}_3$.
\end{proposition}

Our article is structured as follows: In Section~\ref{hypersurfaces}, we review basic hypersurface theory and use  conformal invariance as the organising principle when producing expressions for    
 hypersurface invariants (including a novel {\it hypersurface Bach tensor} $B_{ab}$ which already appears in Proposition~\ref{Calzone} above).
In Section~\ref{dfvc} we develop our hypersurface variational calculus and 
treat Willmore and minimal surfaces as basic examples. In Section~\ref{cfm}
we compute the variational gradient for our four-manifold analog of the Willmore energy. Section~\ref{tractors} is devoted to a review of the essential tractor calculus methods needed to effectively treat the singular Yamabe problem. In Section~\ref{spaces} we compute the obstruction density for four-manifold hypersurfaces. We also compute the obstruction for example metrics, including a hypersurface surrounding the horizon of a Schwarzchild black hole. The appendix gives explicit formul\ae\ for the obstruction density for manifolds of dimensions five  or less  (so including ${\mathcal B}_4$) that can easily be implemented for explicit metrics using computer software packages.

\subsection{Riemannian and conformal geometry conventions}\label{conventions}

We work  with manifolds $M$ of dimension $d$ 
and embedded hypersurfaces of dimension $\db:=d-1$.  
When the dimension $d$ equals 
 three or four,
 we often refer to the latter   as surfaces and spaces, respectively. (Note that the exterior derivative will be denoted by $\dd$, to avoid confusion with the dimension~$d$.)
When~$M$ is equipped with a Riemannian metric~$g$, its Levi-Civita connection will be denoted by  $\nabla$
or $\nabla_a$ in an abstract index notation ({\it cf.}~\cite{Penrose}). The corresponding Riemann curvature tensor $R$  is
$$R(u,v)w=[\nabla_u,\nabla_v]w-\nabla_{[u,v]}w\, ,$$
for arbitrary vector fields $u$, $v$ and $w$. In the abstract index notation, $R$  is denoted by
$R_{ab}{}^c{}_d$ and $R(u,v)w$ is $u^a v^b R_{ab}{}^c{}_d w^d$.  
Cotangent and tangent spaces will be canonically identified using the metric tensor,~$g_{ab}$ and this will be used to raise and lower indices in the standard fashion.

The Riemann curvature can be decomposed into the 
trace-free {\it Weyl curvature} $W_{abcd}$ and the symmetric {\it Schouten tensor}~$\Rho_{ab}$ according to 
$$
R_{abcd}=W_{abcd}+ 2g_{a[c} \Rho_{d]b}-2g_{b[c} \Rho_{d]a}\, .
$$
Here antisymmetrization over a pair of indices is denoted by square brackets so that $X_{[ab]}:=\frac{1}{2}\big(X_{ab}-X_{ba}\big)$. The Schouten and Ricci tensors are related by
$$
\Ric_{bd}:=R_{ab}{}^a{}_d=(d-2)\Rho_{bd}+g_{ab}\J\, ,\quad \J:=\Rho_a^a\, .
$$
The scalar curvature $\Sc=g^{ab}\Ric_{ab}$, thus $\J=\Sc/(2(d-1))$. In two dimensions the Schouten tensor is ill-defined, but we take $\J=\frac12 \Sc$ in that case. In dimension $d\geq 3$, the curl of the Schouten tensor gives the {\it Cotton tensor}
$$
C_{abc}:=2\nabla_{[a} \Rho_{b]c}\, .
$$ 
In dimensions $d\geq 4$, 
\begin{equation}\label{divW}
(d-3) C_{abc}=\nabla_dW_{ab}{}^d{}_c\, .
\end{equation}

A conformal geometry is a manifold equipped with a conformal structure~$\cc$, namely an equivalence class of metrics with equivalence relation $g'\sim g$ when $g'=\Omega^2 g$ for $C^\infty M\ni \Omega>0$. The corresponding Weyl tensors obey $W^{^{\sss g'}}{}_{\!\!\!\!\!\!ab}{}^c{}_d=W^{^{\sss g}}{}_{\!\!\!\!ab}{}^c{}_d$ and thus the  Weyl tensor is a conformal invariant. Conformal invariants may also be density-valued A weight~$w$ density is a  section of the oriented line bundle $\big[(\wedge^d TM)^2\big]^{\frac{w}{2d}}=:\ce M[w]$.
Vector bundles ${\mathcal V}M$ tensored with this are denoted ${\mathcal V}M\otimes \ce M[w]=:{\mathcal V}M[w]$. Each metric $g\in\cc$ determines a section of $(\wedge^d T^*M)^2$ 
and so there is a tautological global section~$\bm g$  of $\odot^2T^*M[2]$ called the {\it conformal metric}.  Hence, each $g\in\cc$ is in~$1:1$ correspondence with a strictly positive section~$\tau\in \ce M[1]$ via $g=\tau^{-2}{\bm g}$. Such sections~$\tau$ will be referred to 
as a {\it true scale}, and we will often present conformally invariant formul\ae\  in terms of a Riemannian metric $g$ by making an (arbitrary) choice of scale $\tau$. It will  be clear from context where this has been done.

In situations where we are dealing with tensors built from higher rank tensors by contractions with vectors or covectors into the first slot such as  $X(u,\cdot)$ (or equivalently $u^a X_{ab}$), we will often use the notation $\stackrel{{\sss u}}X_{a}$. In higher rank cases where more indices are missing, an alternating leftmost-rightmost contraction procedure has been followed, so if $X$ is rank~$3$, then $\stackrel{\sss u}X_b$ denotes $X(u,\cdot,u)$ or equivalently $u^a X_{abc} u^c$.
Contractions across multiple indices of tensors with one another will be denoted by an obvious bracket notation, for example, if $Y$ is rank two and $X$ is rank three, $X(Y,\cdot)$ denotes the covector $X_{abc}Y^{ab}$. Also, for rank two symmetric tensors $S,T$, we will have occasion to employ matrix notations such as $S^2_{ab}:=S^{\phantom{c}}_{ac}S^c_b$
and $\tr ST:=S_{ab}T^{ba}$. Unit weight symmetrization over groups of indices is denoted by round brackets, which we adorn with the symbol $\circ$ when projecting onto the trace-free part of this, for example $X_{(ab)\circ}:=\frac12\big(X_{ab}+X_{ba}\big)-\frac1d g_{ab} X_c{}^c$; for mixed symmetry tensors, projection onto the  trace-free part of a group of indices will be indicated by the notation $\{\ \}\circ$ and their corresponding section spaces will be decorated with a subscript~$\circ$. Finally, we will also use a dot product notation for inner products of vectors $g(u,v)=u_av^b=u.v$ as well as divergences $\nabla.X_b:=\nabla^a X_{ab}$, and $|u|:=\sqrt{u_au^a}:=\sqrt{u^2}$ denotes the length of a vector $u$.

\section{Embedded hypersurfaces}\label{hypersurfaces}

A {\it hypersurface}~$\Sigma$ is a smoothly embedded codimension~1 submanifold of a smooth manifold~$M$. 
A function $s\in C^\infty M$ is called a {\it defining 
function} for $\Sigma$ if the hypersurface is given by its zero locus, {\it i.e.} $\Sigma={\mathcal Z}(s)$ {\it and} the
exterior derivative ${\bm d}s=:n$ is nowhere vanishing along~$\Sigma$. The covector field $n|_\Sigma$ is then a {\it conormal} to the hypersurface and $\hat n^a:=\big(n^a/|n|\big)\big|_\Sigma$ is the {\it unit conormal} while $n^a|_\Sigma$ is a {\it normal vector}. 

The conormal allows the tangent bundle $T\Sigma$ and the subbundle of $TM|_\Sigma$
orthogonal to $n^a$, denoted $TM^\top$, to be canonically identified. Thus we may use the same abstract indices for $T\Sigma$ as for $TM$. The projection of tensors, to hypersurface tensors will be denoted by a superscript~$\top$. Given an extension $\hat n^a$ of the unit normal to $M$, we will also write $v^\top:=v-\hat n\,  \hat n.v$ for vectors (and similarly for tensors) $v\in TM$. Upon restriction, these give hypersurface vectors (and tensors).

The metric $g$ on $M$ induces a metric $\bar g$ on $\Sigma$ given by
\begin{equation}\label{induced}
\bar g_{ab} = g_{ab}|_\Sigma-\hat n_a  \hat n_b\, .
\end{equation}
Given a hypersurface vector field $v^a\in \Gamma(T\Sigma)$, the Gau\ss\ formula relates the restriction to $T\Sigma$ of the Levi-Civita connection $\nabla$ on~$M$  to the hypersurface Levi-Civita connection of $\bar g$ according to 
\begin{equation}\label{II}
\nablab_a v^b = \nabla^\top_a v^b + \hat n^b \II_{ac} v^c\, .
\end{equation} 
In the above formula, the  right hand side is defined independently of how $v^a\in T\Sigma$ is extended to $v^a\in TM$. In general, we will use the term {\it tangential} for operators~${\mathcal O}$ along~$\Sigma$ if they have  the property that for $v$ a smooth extension of some tensor $\bar v$ defined along~$\Sigma$, the quantity ${\mathcal O} v|_\Sigma$ is independent of the choice of this extension.
In the above display, the {\it second fundamental form} $\II_{ab}\in \Gamma(\odot^2 T^*\Sigma)$ is given by
$$
\II_{ab}=\nabla^\top_a \hat n_b\, .
$$
Its ``averaged'' trace is the {\it mean curvature} 
$$
H=\frac1{\bar d\, }\, \II^a_a\, .
$$
We will often use bars to distinguish intrinsic hypersurface~$\Sigma$ quantities from their host space counterparts. In particular, the equations of Gau\ss\  relate ambient  curvatures to intrinsic ones according to
\begin{equation}\label{Gauss}
\begin{split}
 R_{abcd}^\top&=\bar R_{abcd}-2\II_{a[c}\II_{d]b}\, ,\\[1mm]
\Ric_{ab}^\top&=\overline{\Ric}_{ab}+\II^2_{ab}-\db H \II_{ab}\, ,\\[1mm]
\Sc&-2\Ric(\hat n,\hat n)=\overline\Sc-\tr\II^2+\db^2H^2\, . 
\end{split}
\end{equation}
The last of these recovers Gau\ss' {\it Theorema Egregium} for a Euclidean ambient space.
The Codazzi--Mainardi equation gives the covariant curl of the second fundamental form in terms of ambient curvature
\begin{equation}\label{Main}
\nablab_{[a}\II_{b]c}=\frac12 \stackrel{\hat n}R{\!}^\top_{cba}\, .
\end{equation}
Finally the divergence of the Codazzi-Mainardi equation determines the Laplacian of mean curvature:
\begin{equation}\label{LapH}
\bar \Delta H=\frac1\db\big(\nablab.\nablab.\II-\nablab^a\!\stackrel{\hat n}\Ric{\!\!}_{a}^{\!\top}\big)\, .
\end{equation}

\subsection{Conformal hypersurfaces}

The ambient conformal structure $\cc$ induces 
a conformal structure $\bar\cc$ on $\Sigma$.
Locally we may always assume there is a section $\hat n_a\in \Gamma(T^*M[1])$ obeying ${\bm g}_{ab}\hat n^a \hat n^b=1$. This {\it conformal unit conormal} plays the role 
of a Riemannian unit conormal. The trace-free part of the second fundamental form $\IIo_{ab}=\II_{ab}-H\bar g_{ab}$, defined (in a choice of scale) in terms of $\hat n_a$ via  Equation~\nn{II}, is a conformally invariant section of $\odot^2T^*M[1]$.
Thus the quantity
\begin{equation}\label{rigidity}
K:=\tr\IIo^2\in \Gamma(\ce \Sigma[-2])\, ,
\end{equation}
is a conformal density of weight $-2$, that  we call  the {\it rigidity density}.

Conformal invariance of the trace-free second fundamental form~$\IIo$ and the Weyl tensor
can be used to decompose the equations of Gau\ss\  and Codazzi--Mainardi equations into 
conformally invariant parts. First, the Ricci equation in~\nn{Gauss} gives what we shall call the Fialkow--Gau\ss\ equation ({\it cf}.~\cite{YuriThesis})
\begin{equation}\label{Fialkow}
\IIo^2_{ab}\!-\frac12\, \bar g_{ab}\scalebox{1.2}{$\frac{\IIo_{cd_{\phantom{a}\!\!}}\!\IIo^{cd}}{\bar d-1}$}-\Whn_{ab}=(\db-2)\!\Big(\!\Rho_{ab}^\top-\bar \Rho_{ab}+H\IIo_{ab}+\frac 12\,  \bar g_{ab}H^2\!\Big)\!=:\!(\db-2)\F_{ab}\, .
\end{equation}
The left hand side of this equation
is conformally invariant, which proves that the Fialkow tensor $\F_{ab}$ defined by the above in dimension $d\geq 4$ is a conformally invariant section of~$\odot^2T^*M[0]$. This combined with the trace-free second fundamental form gives a new distinguished density 
\begin{equation}\label{mrd}
L:=\tr(\IIo\F)\in \Gamma(\ce \Sigma[-3])\, ,
\end{equation}
which we call  the {\it membrane rigidity density}. 

In dimensions $d\geq 4$, the trace-free part of the first Gau\ss\  equation allows the tangential part of the ambient Weyl tensor to be
traded for its hypersurface counterpart:
$$
W^\top_{abcd}=\bar W_{abcd}-2\IIo_{a[c}\IIo_{d]b}
-\frac{2}{\db-2}\, \big(\bar g_{a[c}\IIo^2_{d]b}-
\bar g_{b[c}\IIo^2_{d]a}
\big)
+\frac{2}{(\db-1)(\db-2)}\, \bar g_{a[c}\bar g_{d]b} K\, .
$$
In dimensions $d\geq 4$, the third Gau\ss\ equation 
expresses the rigidity density in terms of the difference between ambient and hypersurface 
scalar curvatures:
\begin{equation*}\label{JJbar}
K=2(\db-1)\Big(\J- \Rho(\hat n,\hat n)-\bar\J +\frac{\db}{2}\, H^2\Big)
\, .
\end{equation*}
This shows that~$\J-\Rho(\hat n,\hat n)-\bar \J +\frac{\db}{2}\, H^2$ is a conformally invariant weight $-2$ density.

The trace-free part of the Codazzi--Mainardi equation shows that the trace-free curl of the trace-free second fundamental form is a weight $1$ conformal invariant:
\begin{equation}\label{tfcm}
\Wn^\top_{abc}=2\nablab_{\{[c}\IIo_{b]a\}\circ}\, .\end{equation}

Finally, we will need a certain novel four-manifold hypersurface invariant:

\begin{lemma}\label{hsb}
Let $\db=3$ and
$$
\Bach_{ab}^{\sss g}:=\Chn^\top_{(ab)}+H\Whn_{ab}-\nablab^c\Whn^\top_{(ab) c}\, .
$$
Then  $\Bach_{ab}^{\sss g}$ defines a symmetric,  weight $-1$ conformal hypersurface density $$\Bach_{ab}\in\Gamma(\odot^2 T^*\Sigma[-1])$$ given in a choice of scale $g\in\cc$ by $
\Bach_{ab}^{\sss g}$.
\end{lemma}

\begin{proof}
This result can be established by examining the conformal  transformation properties of the Cotton tensor, mean curvature and the divergence of the Weyl tensor: 
\begin{equation}\label{transform}
\begin{split}
C^{\scriptscriptstyle\Omega^2\! g}_{abc}&=\Omega^{-1}(C_{abc}-W_{abc}{}^d\Upsilon_d)^\top\, ,
\quad H^{\sss\Omega^2\! g}=\Omega^{-1}(H+\Upsilon.\hat n)\, ,\quad\\[1mm]
\big[\nablab^c \Whn^\top_{abc}\big]^{\sss\Omega^2 g}&=
\Omega^{-1}\big(\nablab^c \Whn^\top_{abc}
+(\db-4)\Upsilon^c \Whn^\top_{(ab)c}
+(\db-2) \, \Upsilon^c \Whn^\top_{[ab]c}\big)\, ,
\end{split}
\end{equation}
where $\Upsilon_a := \Omega^{-1}\nabla_a \Omega$.
For the Cotton tensor transformation 
we must decompose $\Upsilon_d=\Upsilon_d^\top+\hat n^d \Upsilon.\hat n$,
the second term of which cancels the term in $\Bach_{ab}^{\sss\Omega^2 g}$ produced by the mean curvature. Setting $\db=3$ and keeping only the symmetric part of the Weyl divergence terms gives the final cancellation.\end{proof}

\begin{remark}
Invariance of the density~$\Bach_{ab}$ for almost Einstein structures follows from \cite[Prop\-osition 4.3]{Goal}, where it was related to the Bach tensor. It also appeared in the construction of a tractor analog of the exterior derivative in~\cite{GLW}. Hence we shall call~$\Bach_{ab}$ the {\it hypersurface Bach tensor}.
\end{remark}

\subsubsection{Invariant operators}

Conformally invariant operators are maps between 
sections conformally weighted tensor bundles.  These play a fundamental role in conformal and hypersurface geometries. Quite generally we will be interested in examples of these whose  target and domain have differing weights and possibly tensor types. There are two main examples that are crucial to following developments.

\begin{proposition}\label{BGGop}
Let $d\geq 3$ and $X^{ab}$ be a rank~2, symmetric, weight $-d=-\db-1$, trace-free hypersurface tensor. The mapping, given in a choice of scale by
\begin{equation}\label{BGG}
X^{ab}\longmapsto L_{ab} X^{ab}:=
\left\{
\begin{array}{cl}
\nablab_a\nablab_bX^{ab}+\bar P_{ab}X^{ab}\, ,&\db\geq 3\, ,\\[2mm]
\nablab_a\nablab_bX^{ab}+P^\top_{ab}X^{ab}+H\IIo_{ab}X^{ab}\, ,&\db=2\, ,
\end{array}
\right.
\end{equation}
determines a conformally invariant operator mapping 
$$
\Gamma(\odot_\circ^2 T\Sigma[-d\, ])\xrightarrow{L_{ab}}
\Gamma(\odot_\circ^2 T\Sigma[- d\, ])\, .
$$
\end{proposition}

\begin{proof}
In dimensions $\db\geq 3$, 
the result is standard and germane to any conformal geometry from the theory of invariant differential operators  based on BGG~sequences, see for example~\cite{EastwoodSlovak}. When $\db =2$
%, one can either employ a dimensional continuation argument by using the Fialkow--Gau\ss\ equation~\nn{Fialkow} to show that $ \Rho_{(ab)\circ}^\top+H\IIo_{ab}$  and $\bar \Rho_{(ab)\circ}$ differ only by $(\db-2)$ times the conformally invariant trace-free part of the Fialkow tensor. Alternatively, 
one can choose a scale, and transform the metric explicitly under 
$g_{ab}\to\Omega^2 g_{ab}$ and then directly verify conformal invariance
using the mean curvature transformation given in Equation~\nn{transform} as well as those for the  Schouten tensor and a weight $w$ 
vector~$v^a$ (which extends to higher rank tensors by linearity):
\begin{equation*}
\begin{split}
\Rho^{\scriptscriptstyle\Omega^2 \!g}_{ab}\ &=\Rho_{ab}^{{\sss g}}-\nabla_{a}\Upsilon_b+\Upsilon_a\Upsilon_b-\frac12 \Upsilon^2 g_{ab}\, ,\\
\nabla_b^{\scriptscriptstyle\Omega^2\!g}v^a
&=\Omega^{w}\big(
\nabla_b^{\sss g} v^a+(w+1)\Upsilon_b v^a -\Upsilon^a v_b+\delta_b^a\Upsilon.v
\big)\, .
\end{split}
\end{equation*}
\end{proof}

Observe that $L^{ab}$ is well-defined when acting on the conformal density $\IIo^2_{(ab)\circ}$; we record an identity for this  required in future developments:
\begin{proposition}\label{LIIo2o}
Along hypersurfaces in a four-manifold one has
$$
\scalebox{.89}{$
L^{ab}\IIo^2_{(ab)\circ}
=\frac 13 (\nablab^a\IIo^{bc})(\nablab_a \IIo_{bc})+
\frac12(\nablab.\IIo_a)(\nablab.\IIo^a)
+\frac23\IIo^{ab}\!\bar\Delta\IIo_{ab}
-\frac23\bar \J K
 - \frac43\IIo^{ab}\nablab^c\Whn{\!}_{abc}^\top+\frac12\Whn^{abc}\Whn{\!}_{abc}^\top$}
\,  .
$$
\end{proposition}

\begin{proof}
The proof is a combination of basic tensor algebra and repeated usage of the trace-free Codazzi--Mainardi Equation~\nn{tfcm} applied to the double divergence of $\IIo^2_{(ab)\circ}$.  We record the main steps below:
\begin{equation*}
\begin{split}
\nablab^a\nablab^b&\IIo^2_{(ab)\circ}=
2\nablab^a(\IIo_{c(a}\nablab^b\IIo^{c}_{b)})-\frac13\bar\Delta K\\[1mm]
=&\,(\nablab.\IIo_a)(\nablab.\IIo^a)+2\IIo^{ab}\nablab^c\nablab_a\IIo_{bc}+(\nablab_a\IIo_{bc})(\nablab^b\IIo^{ac})+\IIo^{ac}[\nablab_a,\!\nablab_b]\IIo_{c}^b\!-\frac13\bar\Delta K\\[1mm]
=&\,\frac13(\nablab^a\IIo^{bc})(\nablab_a\IIo_{bc})
+\frac12(\nablab.\IIo_a)(\nablab.\IIo^a)
-\frac 23\IIo^{ab}\bar\Delta\IIo_{ab}
+(\nablab^a\IIo^{bc})\Whn^\top_{cab}\\&
+2\IIo^{ab}\nablab^c\nablab_a\IIo_{bc}
+\IIo^{ac}[\nablab_a,\!\nablab_b]\IIo_{c}^b\\[1mm]
=&\, \frac13(\nablab^a\IIo^{bc})(\nablab_a\IIo_{bc})
+\frac12(\nablab.\IIo_a)(\nablab.\IIo^a)
+\frac 23\IIo^{ab}\bar\Delta\IIo_{ab}
+\frac12\Whn^{cab}\Whn^\top_{abc}\\&
-\frac43\IIo^{ab}\nablab^c\Whn^\top_{abc}
+\frac13\IIo^{ac}[\nablab_a,\!\nablab_b]\IIo_{c}^b\, .
\end{split}
\end{equation*}
The third line was achieved  using  identity
\begin{equation}\label{nemo}
\IIo^{ab}\nablab_a\nablab.\IIo_b=
\frac23\, \IIo^{ab}\bar\Delta\IIo_{ab}-\frac23\bar \J \, K -2\tr (\IIo^2\bar P)-\frac23 \IIo^{ab}\nablab^c \Wn_{bac}^\top\, ,
\end{equation}
which can be proved by using
Equation~\nn{tfcm} to rewrite $\nablab^c \nablab_a \IIo_{bc}=\bar \Delta \IIo_{ab}-\nablab_b \nablab. \IIo_{a}+\mbox{\it lower order terms}$.
In three dimensions, the Weyl tensor vanishes so the last term in the computation of~$\nablab^a\nablab^b\IIo^2_{(ab)\circ}$ displayed above Equation~\nn{nemo} can only contribute intrinsic Schouten and scalar curvature terms; these are precisely $-\bar\Rho^{ab}\IIo^2_{(ab)\circ}-\frac23\bar \J K$ and this completes the proof.
\end{proof}

The second conformally invariant hypersurface operator needed implements the Robin 
 combination of Neumann and Dirichlet boundary conditions and was shown to be conformally invariant by Cherrier: 
 \begin{definition}\label{ROBIN}
 Given   a weight~$w$ conformal density 
 $f\in\Gamma(\ce M[w])$, we define the {\it (conformal) Robin} operator
  \begin{equation}\label{Robin}
 \delta_{\hat n}  f :=\nabla_{\hat n}f\big|_\Sigma - wH \cdot f\big|_\Sigma\, .
\end{equation}
\end{definition}

\begin{remark}
The operation $(\nabla_{\hat n}-\frac{w}{d-1}\nabla.\hat n)f$ is the Lie derivative
of the density $f$ along the unit  normal vector $\hat n$, so by construction $\delta_{\hat n} f$ is a weight $w-1$ conformal hypersurface density. 
% Later we shall see that the solution to the singular Yamabe problem provides  canonical extensions for certain conformal hypersurface invariants. In that context we will denote the operator~$\delta_{\hat n}$ by $\delta_{\rm R}$. In both contexts we will refer to these operators as a {\it Robin operator}.
\end{remark}

The Robin operator can be extended to also act on conformal tensor densities, in what follows we will need the case of a rank two, symmetric conformal tensor density:

\begin{lemma}
Suppose $X_{ab}\in \Gamma(\odot^2T^*M[w])$ obeys $\hat n.X_b|_\Sigma=0$ 
%and is an extension of $\bar X_{ab}\in\Gamma(\odot^2T^*\Sigma[w])$  
then one has a conformally invariant operator given here in a scale $g\in \cc$ 
\begin{equation}\label{deltaRab}
\delta_{\hat n} X_{ab}:=\big(\nabla_{\hat n} X_{ab} - (w-2)H X_{ab}\big)^{\!\!\top}\big|_\Sigma\, .\end{equation}
This  defines  a section  of $\Gamma(\odot^2 T^*\Sigma[w-1])$ in the  corresponding induced scale $\bar g\in \bar \cc$. 
\end{lemma}

\begin{proof}
This is a member of a general class of results that can be obtained 
from the Laplace-Robin operator discussed in Section~\ref{LRO}. A simple proof of the statement uses the conformal transformation property of the mean curvature given in Equation~\nn{transform} and that for the normal derivative of a rank two tensor density:
$$
\nabla^{\sss \Omega^2 g}_{\hat n} X_{ab}=
\Omega^{w-1}\Big(
\nabla_{\hat n} X_{ab}+(w-2)\Upsilon.\hat n X_{ab}-
2\Upsilon_{(a}\hat n.X_{b)}
+2\hat n_{(a}\Upsilon.X_{b)}\Big)\, .
$$
Since $\hat n.X_b|_\Sigma=0$, and the formula for $\delta_{\hat n}$ projects onto the tangential part of $\nabla_{\hat n} X_{ab}$, only the first two terms on the right hand side above survive, and the result thus follows. 
\end{proof}

\section{Defining function variational calculus}\label{dfvc}

To handle hypersurface variational problems
we  treat hypersurfaces in terms of defining functions.
Locally any hypersurface is the zero locus of  some defining function, so no generality is lost studying hypersurfaces $\Sigma={\mathcal Z(s)}$ for some defining function~$s$. For simplicity we take~$M$ oriented with volume form~$\omega$; this and the conormal $\dd s$ determine the orientation of $\Sigma$. We also shall assume, again mostly for simplicity, that $\Sigma$ is compact.  

Hypersurface invariants
are defined as follows (see~\cite[Definition 2.6]{GW15}):
\begin{definition}\label{R-invtdef}
%For a hypersurface~$\Sigma$, 
A {\em scalar Riemannian pre-invariant}  is a mapping~$\Pre$ which assigns a smooth function $\Pre(s;g)$ to  each Riemannian~$d$-manifold
$(M,g)$ and  hypersurface defining function~$s$, satisfying: \\ 
\begin{enumerate}
\item[(i)] $\Pre(s; g)$ is natural;
for any diffeomorphism $\phi:M\to M$ we have $$\Pre(\phi^* s,\phi^* g ) =
\phi^* \Pre(s; g)\, .$$
\item[(ii)] If~$s'=vs$ for some smooth, positive function~$v$  then $$\Pre(s; g)|_\Sigma= \Pre(vs;
g)|_\Sigma\, ,\mbox{ where }\Sigma={\mathcal Z(s)}\, .
$$
\item[(iii)] $\Pre$ is given by a universal polynomial expression such that, 
given a local coordinate system~$(x^a)$ on~$(M,g)$,~$\Pre(s,u; g)$ is given by
a polynomial in the variables
$$g_{ab},~\partial_{a_1}g_{bc}, ~\cdots, ~\partial_{a_1}\partial_{a_2}\cdots \partial_{a_k}
g_{bc},~(\det g)^{-1},\omega_{a_1\ldots a_d}\, ,$$
$$
%u,~\partial_{b_1} u,~\cdots~,\partial_{b_1}\partial_{b_2}\cdots \partial_{b_\ell} u,\, 
s,~\partial_{b_1} s,~\cdots~,\partial_{b_1}\partial_{b_2}\cdots \partial_{b_\ell} s, 
~| \boldsymbol{d} s |_g^{-1}, $$ for some positive integers~$k,\ell$. Here $\partial_a$ means~$\partial/\partial x^a$,
$g_{ab}=g(\partial_a,\partial_b)$,~$\det g= \det(g_{ab})$ and $\omega_{a_1\ldots a_d}=\omega(\partial_{a_1},\ldots,\partial_{a_d})$. 
\end{enumerate}
A {\em  scalar Riemannian invariant} of a hypersurface~$\Sigma$ is the
restriction~$\bar \Pre(\Sigma;g):=
%\Pre(\Sigma;g):=
\Pre(s;g)|_\Sigma$ of a  
pre-invariant~$\Pre(s;g)$ to~$\Sigma$.
%:= \mathcal{Z}(s)$.
\end{definition}

\begin{remark}
A hypersurface invariant $\bar \Pre$ depends only on the data $(M,g,\Sigma)$.
 For point~(i) note that if~$\Sigma=\mathcal{Z}(s)$, then~$\phi^{-1} (\Sigma)$
is a hypersurface with defining function~$\phi^* s$. 
In (ii), the requirement $s'=vs$ means that $s'$ and $s$ are
are two
compatibly oriented defining functions with 
$\mathcal{Z}(s)=\mathcal{Z}(s')=:\Sigma$.
In~(iii), since we are ultimately interested in hypersurface invariants, there is no loss of generality studying defining functions such that  $| \boldsymbol{d} s |_g^{-1}\neq 0$ 
everywhere in $M$, since we may, if necessary replace $M$ by a local neighborhood of $\Sigma$.

%In (iii)~$\partial_a$ means~$\partial/\partial x^a$,
%$g_{ab}=g(\partial_a,\partial_b)$,~$\det g= \det(g_{ab})$ and $\omega_{a_1\ldots a_d}=\omega(\partial_{a_1},\ldots,\partial_{a_d})$. 

%The conditions (i),(ii) and
%(iii) mean that any Riemannian invariant~$\Pre(s;g)|_\Sigma$ of~$\Sigma$, is
%entirely determined by the data~$(M,g,\Sigma)$, and this justifies the notation~$\Pre(\Sigma;g)$. Then in this notation the naturality condition of (i) implies
%~$\phi^*\!( \Pre(\Sigma,g))= \Pre(\phi^{-1}(\Sigma),\phi^* g)$. Also, 
The 
definition  extends to  {\em tensor-valued 
  hypersurface pre-invariants} and {\em invariants} by considering
 tensor-valued functions~$\Pre$ and requiring the
coordinate components of the image satisfy  conditions~(ii),~(iii), and the obvious adjustment of (i).  The term {\em hypersurface invariant} will be taken to mean either
tensor or scalar valued hypersurface invariants.
\end{remark}

Local Riemannian 
invariants, constructed in the usual way, provide 
trivial examples of  hypersurface preinvariants. Examples of preinvariants that depend non-trivially on both $s$ and $g$ are given below:

\begin{example}
Given a hypersurface defining function~$s$, 
 examples of preinvariants  are 
$$
\Pre_a(s;g)=\frac{\nabla_a s}{|\nabla s|} \ \mbox{ and } \ 
\Pre_{ab}(s;g)=\nabla_{a}\left(\frac{\nabla_b s}{|\nabla s|}\right)
-\frac{\nabla_a s}{|\nabla s|}
\frac{\nabla^c s}{|\nabla s|}\  \nabla_c\left(\frac{\nabla_b s}{|\nabla s|}\right)
%\!
%\mbox{ and } P_a(s,u;g)=\frac{u}{|\nabla s|}
\, .
$$
Upon restriction to $\Sigma={\mathcal Z}(s)$,  these give the unit conormal and second fundamental form hypersurface invariants.
 It is important to note that for distinguished choices of defining function~$s$, these expressions simplify considerably. For example, if $s$ obeys $|\nabla s|=1$, then the second fundamental form is $\II_{ab}=(\nabla_a \nabla_b s)|_\Sigma$, but the quantity $\nabla_a \nabla_b s$ is, of course, {\it not} a preinvariant.
\end{example}

\medskip

We now consider integrals $\int_\Sigma dA_{\bar g}\, \raisebox{.5mm}{\scalebox{.6}{$\bullet$}}$ over a hypersurface, where $dA_{\bar g}$ is the volume element of the induced metric $\bar g_{ab}$. Let~$s$ be a defining function for $\Sigma$. Then, since (see Equation~\nn{induced})
$$
\bar g = \Big(g - \frac{\dd s\otimes \dd s}{|\dd s|^2}\Big)\Big|_\Sigma\, ,
$$
%\edz{RS: Notn should be made consistent here\\ Also if indices are abstract then $g=g_{ab}$ -- so reader {\em should} be confused by the $dx^a$'s here}
%Thus 
%$$
%g:=g_{ab}dx^a dx^b = \gamma_{ab} dx^a dx^b + \Big(\frac{n. dx}{|\nabla s|}\Big)^2
%\stackrel\Sigma =\bar g +\Big(\frac{ds}{|\nabla s|}\Big)^2\, .
%$$
%\edz{RS: This is correct but could confuse.}
%Hence
choosing coordinates 
$(x^a)=(s,y^i)$ where $g(ds,dy^i)|_\Sigma=0$, it follows that
%along~$\Sigma$ it follows that
$$
\sqrt{\det g\, }\, \Big|_\Sigma=
\frac{\,\sqrt{\det \bar g^{^{\phantom{\!o}}}}\ }{  \; \; \;  |{\bm d} s|\big|_\Sigma}\, .
$$ 
   Thus we have the following:
\begin{proposition}
Let $\Sigma$ be a smoothly embedded, compact hypersurface in $M$, with defining function~$s$,  and let 
$\bar \Pre$ be a hypersurface invariant for $\Sigma$ with preinvariant $\Pre(s;g)$.
%$f\in C^\infty M$.
Then 
\begin{equation}\label{holintegral}
\int_\Sigma dA_{\bar g}\, \bar \Pre\, = %\int_M\sqrt{ \det g}\  |\nabla s|\,  {\bm \delta}(s) \, f^{\rm ext}=: 
\int _M dV_g  \, |\nabla s|\,  {\bm \delta}(s) \, \Pre(s;g)\, ,
\end{equation}
where $\bar g$ is the pullback  
%/restriction of
of the metric $g$  to $\Sigma$.  
\end{proposition}
%\noindent
In the above 
%$f^{\rm ext}$ is any extension of the function $f$ on $\Sigma$ to $M$ and 
${\bm \delta}(s)$ is the Dirac delta distribution. The above formula is   standard, see for example~\cite{Osher} for its use in the theory of implicit surfaces. This has the advantage of allowing us to perform surface integrals over hypersurface invariants $\bar \Pre$ in terms of  (often far simpler) extensions to $M$. 

\begin{remark}\label{predist}
The distributional identities $x{\bm\delta}(x)=0$ and  $
{\bm \delta}(f(x))=\sum_{x_0\in{\mathcal Z}(f)} \frac{\scalebox{.7}{$\bm \delta$}(x-x_0)}{|f'(x_0)|}$ for functions $f$ with simple roots,
imply $|\nabla(vs)|{\bm \delta(vs)}=|\nabla s| {\bm \delta}(s)$. Thus, the integrand 
$|\nabla s| {\bm \delta}(s)\Pre(s;g)$ on the right hand side of~\nn{holintegral} can be loosely viewed as distributional hypersurface preinvariant.
\end{remark}

\medskip

Moving now to 
variational considerations, let 
 $\bar \Pre$ be a hypersurface invariant. The variation of the 
functional $I(\Sigma)=\int_\Sigma dA_{\bar g} \, \bar \Pre$, 
with respect to the embedding of the hypersurface $\Sigma$,
is defined by
\begin{equation}\label{vary}
\delta I:=\frac{dI(\Sigma_t)}{dt}\Big|_{t=0}=\frac{d}{dt} \int_{\Sigma_t} dA_{\bar g_t}\,  \bar \Pre(\Sigma_t;g)\, \Big|_{t=0}\, .
\end{equation}
Here $\Sigma_t$ denotes a smooth, one parameter family of hypersurfaces such that~$\Sigma_0=\Sigma$ and~$\Sigma_t = \Sigma$ outside a compactly supported subset of $M$. Also,~$\bar g_t$ is the pullback of the metric~$g$ to~$\Sigma_t$.

Our  strategy is, without loss of generality,  to consider  a family of hypersurfaces $\Sigma_t={\mathcal Z}(s_{_t})$ 
where $s_{_t}$ is the family of defining functions
\begin{equation}\label{st}
s_{_t}:=s+tu
\end{equation}
for some smooth, compactly supported $ \delta s:=u\in C^\infty \!M$.
Thus $I(\Sigma_t)=I({\mathcal Z}(s_t))$.
In general we will use~$\delta$ to denote the operation $f(s)\mapsto \big(\frac{df(s+tu)}{dt} \big)\big|_{t=0}=:\delta f$.

Using Equation~\nn{holintegral}
the variational formula~\nn{vary} becomes
$$
\delta I = \frac{d}{dt} \int_{M} dV_g\, 
|\nabla s_{_t}| \, {\bm \delta}(s_{_t}) \Pre(s_{_t};g)\, \Big|_{t=0}\, .
$$
This reformulation leads to the following Lemma which is needed for our key variational Theorem~\ref{VARY}, and is also a useful computational tool:

\begin{lemma}\label{vary_dA}
Let $\Pre$ be a hypersurface preinvariant with corresponding hypersurface invariant~$\bar \Pre$.
Consider a 
  compactly supported variation~$\delta s$  of a  defining function $s$ for a hypersurface~$\Sigma$. Then
%  \edz{Change delta sbar using $\measure$}
%  HERE WE WANT TO DISPLAY THE QUASI-PREINVARIANT (a la last line of proof mores u's than $\delta s$)
\begin{equation}\label{blabla}
\delta \left(\int_\Sigma dA_{\bar g}  \, \bar \Pre\right)=
\int_M dV_g\,  |\nabla s|{\bm \delta}(s)\, \Big[ \delta \Pre - \frac{\delta s}{|\nabla s|}(\nabla_{\hat n} + \nabla.\hat n) \Pre\Big]\, ,
%\int_\Sigma dA_{\bar g}\left[ \, \big(\delta \Pre\big)\big|_\Sigma -%\int_\Sigma dA_{\bar g} 
%\measure
%\,   \delta_{\hat n} \bar \Pre\right]\, ,
\end{equation} 
where $\hat n:=(\nabla s)/|\nabla s|$.
%where 
%%$\delta \bar s:=\delta s|_\Sigma$ and
%$$
%\delta_{\hat n}\Pre:=
%(\nabla_{\hat n}\Pre)|_\Sigma + \db\,  H_{\bar g} \bar \Pre\, .
%$$
\end{lemma}

\begin{proof}
Let $\chi$ be a (fixed) compactly supported cutoff function taking the value 1 on an open  neighborhood of the support of $\delta s$. Then
\begin{equation}
\begin{split}
\delta \left(\int_\Sigma dA_{\bar g}\bar  \Pre\right)&=\delta \left(\int_M dV_g\, \chi\,|\nabla s|\,  {\bm \delta}(s) \,   \Pre\right)\\[1mm]
&=\int_M dV_{g}\, \chi\, \Big[
(\nabla_{\hat n} \delta s) {\bm \delta}(s) \,   \Pre
+|\nabla s|\, {\bm \delta}'(s)\, \delta s\, \Pre
+|\nabla s|\, {\bm \delta}(s)\, \delta  \Pre\Big]\\[1mm]
&=\int_M dV_{g}\, \chi\, \Big[
-(\nabla.{\hat n}) \delta s\,  {\bm \delta}(s) \,   \Pre
-{\bm \delta}(s)\, \delta s\, \nabla_{\hat n} \Pre 
+|\nabla s|\, {\bm \delta}(s)\, \delta  \Pre\Big]\, .
\end{split}
\end{equation}
%In the above $\hat n_a:=\nabla_a s/|\nabla s|$. 
Here we used that the derivative of $\chi$ vanishes where $\delta s$ has support and
$\nabla_{\hat n} {\bm \delta}(s)=|\nabla s| {\bm \delta}'(s)$. Also, here and in the above, ${\bm \delta}'(s)$ can be interpreted in the usual distributional sense. 
In the second line above we  performed an integration by parts using  that the integrand has compact support. The result follows using Equation~\nn{holintegral}.
\end{proof}

\begin{remark}
When $\Pre$ (and thus also $\bar \Pre$) is a density of weight $-\db$, the integral $\int_\Sigma dA_{\bar g} \bar \Pre$ is conformally invariant. In this case the operator appearing in the second term of Equation~\nn{blabla}  is  the Lie derivative of $\Pre$ along~$\hat n$. Along~$\Sigma$ this becomes
$$
\Big((\nabla_{\hat n}  +\nabla.\hat n )\Pre\Big)\Big|_\Sigma=\delta_{\hat n}  \Pre\, ,
$$
where
$\delta_{\hat n}$ is Cherrier's conformally invariant Robin operator; see Definition~\ref{ROBIN}. This  motivates the notation of the definition that follows.
 \end{remark}
 
 \begin{definition}
Let $\Sigma$ be a hypersurface and   let $f\in C^\infty M$ (or  a density of any weight). Then we define
 $$
 \delta_{_{k}}  f = \nabla_{\hat n} f\big|_\Sigma -k\,  H \bar f\, ,
 $$
 where $\bar f=f|_\Sigma$ and $k\in{\mathbb C}$.
 \end{definition}
 
 To compute the varational gradient of $\int_\Sigma dA_{\bar g} \, \bar \Pre$, it remains to express the first term on the right hand side of Equation~\nn{blabla} in the 
form $ \int_M dV_g\,  |\nabla s|{\bm \delta}(s)\, \hat \delta s \, 
%\frac{\delta s}{|\nabla s|} 
\Pre^\prime$
with some local formula for~$\Pre^\prime$ where
 $$\hat\delta s:=\frac{\delta s}{|\nabla s|}\, .$$ 
 We have introduced the  quantity
 $\hat \delta s$ for  homogeneity reasons: 
 Observe that the integral~$\int_{\Sigma_t} dA_{\bar g}\bar  \Pre$ being differentiated on the left hand side of~\nn{blabla} is independent of the
defining function~$s_{_t}$, so in particular is unchanged upon replacing $s_{_t}$ with $v  s_{_t}$ for any positive, smooth function $v$. Thus $  \delta\big(\int_{\Sigma} dA_{\bar g}\bar \Pre\big)=:\delta I(s,\delta s) =\delta I(vs,v\delta s)$.
It is easy to verify the identity
$$
\left.\frac{v\delta s}{|\nabla(vs)|}\right|_\Sigma=\left.\frac{\delta s}{|\nabla(s)|}\right|_\Sigma\, ,
$$
so that along $\Sigma$, $\hat \delta s$ enjoys the same homogeneity property as~$\delta I$.
To rewrite~$\delta I$ with $\hat \delta s$
%$\delta s/|\nabla s|$ 
undifferentiated,  we will need the following ``integration by parts'' lemma:
 \begin{lemma}\label{parts}
Let $s$ be a defining function and suppose $C^a\in \Gamma (TM)$. 
Moreover let $B$ be a test function ({\it i.e.}, smooth with compact support).
 Then
 $$
 \int_M dV_g \, |\nabla s| {\bm \delta}(s)\, C^a \nabla_a^\top B = - \int_M dV_g \, |\nabla s| {\bm \delta}(s)
 B\,  \nabla_a^\top C^\top{}^a\, ,
 $$ 
 where  $\nabla^\top = \nabla -\hat n \nabla_{\hat n}$, $\hat n_a=(\nabla_a s)/|\nabla s|$ and  $C^\top:=C-\hat n\,  \hat n.C$.
\end{lemma}

\begin{proof}
First, not forgetting the dependence of $\nabla^\top$ on $\hat n$, we use that $B$ is a test function and integrate by parts to  compute
$$
\int_M dV_g \, |\nabla s| {\bm \delta}(s)\, C^a \nabla_a^\top B =
-\int_M dV_g \, B\Big[\nabla_a^\top(|\nabla s|\, {\bm \delta}(s) C^\top{}^a)-
|\nabla s|\, {\bm \delta}(s) C^\top{}^b{} \nabla_a^\top(\hat n^a \hat n_b)
\Big]\, .
$$
Next,  a straightforward computation shows that
$$
\nabla_{\hat n} \hat n_a=\frac{\nabla^\top_a |\nabla s|}{|\nabla s|}\, .
$$
Using the distributional identity $\nabla^\top {\bm \delta}(s)=0$, the result follows.
\end{proof}

\begin{remark}
When $C^a$ and $B$, respectively,  take values in  sections of some tensor bundle over $M$ and its dual and these are appropriately contracted to form a scalar, the natural extension of the above Lemma follows immediately. 
\end{remark}
 
The variation $\delta \Pre$ of a preinvariant is linear in the variation $u=\delta s$ and its derivatives up to 
%~$\nabla u$, $\nabla^2 u,\ldots, \nabla^k u$ 
for some finite order~$k$. If all derivatives on the variation $u$ are tangential derivatives~$\nabla^\top u$, $(\nabla^\top)^2\,  u,\ldots, (\nabla^\top)^k \, u$, we can use the above lemma to integrate these by parts.
In fact, as we shall prove in Theorem~\ref{VARY}, this is always the case. 
That is, we show there is a well-defined notion for the 
 {\it functional derivative} $\frac{\delta \Pre}{\delta s}$ of a preinvariant determined by
 \begin{equation}\label{functional}
\int_M dV_g \, |\nabla s|\, 
{\bm \delta}(s) \, \delta \Pre
:=\int_M dV_g \, |\nabla s|\, 
{\bm \delta}(s) \,
\hat \delta s
%\frac{\delta s}{|\nabla s|}
\ \frac{\delta \Pre}{\delta s}
\, ,
\end{equation}
for compactly supported variations~$\delta s$.
We  now give our main variational theorem:
 
 \begin{theorem}\label{VARY}
Let $\Sigma$ be a hypersurface with defining function~$s$ and 
let $\bar \Pre$ be a hypersurface invariant with corresponding preinvariant $\Pre$. Then for compactly supported  variations,  
$$\delta \int_\Sigma dA_{\bar g}
\bar \Pre=
\int_M dV_{g}\,
|\nabla s|{\bm \delta}(s)\,
\hat \delta s 
%\frac{\delta s}{|\nabla s|}
\, \Big( \frac{\delta  \Pre}{\delta  s}-(\nabla_{\hat n}+\nabla.\hat n) \Pre)\Big)\, ,
$$
where $\hat n=(\nabla s)/|\nabla s|$.
 \end{theorem}
 
 \begin{proof}
 Thanks to Lemma~\ref{vary_dA}, we need only  verify the well definedness of $\frac{\delta \Pre}{\delta s}$ in Equation~\nn{functional}. Our first step is to construct coordinates~$(s,x^i)$ where $i=1,\ldots,\db$ such that the vector field $\frac{\partial }{\partial s}$ is orthogonal to constant $s=\varepsilon$ hypersurfaces $\Sigma^\varepsilon:={\mathcal Z}(s-\varepsilon)$.
 %, $\varepsilon \in {\mathbb R}$. 
 This is achieved by integrating the vector fields $g^{-1}({\bm d} s,\cdot)$ and using the flow along their integral curves to push forward coordinates $x^i$ on $\Sigma=\Sigma^{\varepsilon=0}$ to $\Sigma^{\varepsilon}$.
 
 Since $\delta s$ has compact support,  we may write the variation $\delta s=u\in C^\infty M$ in the coordinates $(s,x^i)$ as
 $$u(s,x^i)=u(0,x^i)+s \, U(s,x^i)\, ,$$
 for some smooth, bounded, function $U$.
 Now we decompose the 
 one parameter family of defining functions in Equation~\nn{st} according to 
 $$
 s_{_t} = s^\prime  + t u_0\, ,
 $$
 where the function $u_0$ is defined in these coordinates by $u_0=u(0,x^i)$.
Now note that $1+t\, U(s, x^i)$ is a positive function for $t$ small, because $U$ is bounded.
(So  $s^\prime = s(1+t\, U(s, x^i))$
 is also a defining function.) Thus for all such $t$ we have that 
$I(\Sigma):=\int_{\Sigma} dA_{\bar g}\bar \Pre$ 
 satisfies
 $$I(\Sigma)=I({\mathcal Z}(s))=I({\mathcal Z}(s^\prime))\, .$$ Now
 from the construction of $I$ it follows that $I(t_1,t_t):=I({\mathcal
   Z}(s+t_1 s U + t_2 u_0))$ is differentiable as a function on
 ${\mathbb R}^2$. From this and the previous observation it follows
 that
 $$
 \delta I = \frac{d I(s+t u_0)}{dt}\Big|_{t=0}\, . 
 $$

Next, from the definition of a preinvariant it follows that the
variation $\delta I$ is a sum of terms of the form
 $$
 \int_M dV_g \, |\nabla s|{\bm \delta}(s)\, 
C^{a_1\ldots a_k}\nabla_{a_1}\cdots \nabla_{a_k} u\, ,
 $$
 for some smooth tensors $C^{a_1\cdots a_k}(s;g)$. However, the above coordinate argument shows that we may replace this by 
$$
 \int_M dV_g \, |\nabla s|{\bm \delta}(s)\, 
C^{a_1\ldots a_k} \nabla_{a_1}\cdots \nabla_{a_k} u_0\, .
 $$ 
 However, since $\nabla_{\hat n} u_0=0$, we have 
 $\nabla u_0 = (\nabla-\hat n \nabla_{\hat n})u_0=:\nabla^\top u_0$. An easy induction shows that we can reexpress the integral as a sum of terms 
$$
 \int_M dV_g \, |\nabla s|{\bm \delta}(s)\, 
\tilde C^{a_1\ldots a_\ell} \nabla^\top_{a_1}\cdots \nabla^\top_{a_\ell} u\, ,
 $$ where we written $u$ not $u_0$ because $\nabla^\top s=0$ and
$s{\bm \delta}(s)=0$ (and $\tilde C^{a_1\ldots a_\ell}$ are some
smooth tensors).  Integrating all $\nabla^\top$'s by parts according
to Lemma~\ref{parts} then establishes Equation~\ref{functional}.
 \end{proof}

\begin{corollary}
Let~$\Sigma$ be a hypersurface with defining function~$s$ and let  $\bar \Pre$ be a hypersurface invariant   with corresponding preinvariant~$\Pre$. Then 
$$ \frac{\delta  \Pre}{\delta  s}\Big|_\Sigma-\delta_{_{\!\text{-}\bar d}}\,  \Pre
\ \mbox{ and }\  \frac{\delta  \Pre}{\delta  s}-(\nabla_{\hat n}+\nabla.\hat n) \Pre$$
are a hypersurface invariant  and its corresponding preinvariant. 
\end{corollary}

 \begin{example}\label{copre}
 Let $\hat n_a=\nabla_a s/|\nabla s|$  where $s$ is a defining function for~$\Sigma$. Then
$$\Pre(s;g)=\nabla^a \hat n_a$$
is a preinvariant for $\db$ times the mean curvature. Noting that 
$$\delta \hat n_a =\frac{\nabla_a^\top \delta s}{|\nabla s|}\, ,$$ it follows that
\begin{equation*}
\begin{split}
\delta \Pre &= \nabla^a\Big(\frac{ \nabla^\top_a \delta s}{|\nabla s|}\Big)
%\\[1mm]&
=
\big[\Delta^{\!\top}
-\big|\nabla_a^\top \log |\nabla s|\big|^2
+(\Delta^{\!\top} \log|\nabla s|)
\big]\, \hat \delta s
%\frac{\delta s}{|\nabla s|}
\, ,
\end{split}
\end{equation*}
where $\nabla_a^\top:=\nabla_a - \hat n_a \nabla_{\hat n}$ and $\Delta^{\!\top}:=g^{ab}\nabla^\top_a\nabla_b^\top$.
Notice~$\delta \Pre$ does not contain normal derivatives $\nabla_{\hat n}$ acting on $\hat \delta s$. Thus
$$
\frac{\delta P}{\delta s}=
\Delta^{\!\top} \log|\nabla s|-\big|\nabla_a^\top \log |\nabla s|\big|^2\, .
$$
This is not a preinvariant, but
it is not difficult to verify that
$\frac{\delta \Pre}{\delta s}-(\nabla_{\hat n}+\nabla.\hat n)\Pre$ is and along~$\Sigma$ this equals the hypersurface invariant
$$
-\tr\II^2+\db^2 H^2+\Ric(\hat n,\hat n)=\Sc-\Ric(\hat n,\hat n)-\overline \Sc\, .
$$
Here the last equality relied on Equation~\nn{Gauss}.
 \end{example}

\subsection{Variational methods and  identities}

We now 
explain our strategy for computing variations and
collect some key identities. 
Using Theorem~\ref{VARY}, hypersurface variations are given in terms of preinvariants.
 To express preinvariants in terms of standard hypersurface invariants, we take advantage of property~(ii) of Definition~\ref{R-invtdef} which allows us to choose any defining function when evaluating preinvariants along~$\Sigma$. In a computational context we are often  concerned with expressions built from a sum of terms which separately are not preinvariants 
even though their sum is. 
Therefore in calculations it is important to use 
the  same choice of defining function  for every term.

Given any hypersurface in a Riemannian manifold there is locally a defining function~$s$ 
satisfying
$$
|\nabla s|=1\, .
$$
Defining functions with this property are dubbed {\it unit defining functions}. Since unit defining functions are determined 
by the hypersurface embedding, they provide a good choice for computations.
In fact, there is   a
recursive solution to the formal problem of finding a smooth  unit defining function~$s_1$ given a defining function~$s$ for a hypersurface such that 
$$|\nabla s_1|=1+s^\ell f\, ,$$
for $\ell\in {\mathbb Z}_{\geq 1}$  arbitrarily large and $f\in C^\infty M$ (see~\cite[Section 2.2]{GW15}).
This explicitly relates the jets of~$s$ to hypersurface invariants.
We employ the notation $\stackrel 1=$ for equalities relying on this choice while $\=$  is used when evaluating quantities along $\Sigma$.

In general we abuse  notation  by using the same symbols for preinvariants as for their corresponding hypersurface invariant when the former has been explicitly chosen.
Let us tabulate some commonly used choices for preinvariants as well as their
expressions in terms of a   unit defining function:

\begin{center}
\begin{tabular}{|c|c|c|}
\hline
$\stackrel1=$ &
Preinvariant & Invariant\\
\hline\hline
$n_a:=\nabla_a s$ &$\frac{\nabla_a s^{\phantom{a^a\!\!\!}}}{|\nabla s|}$
& $\hat n_a$
\\[2mm]
$g_{ab}-n_a n_b$ &
 $\gamma_{ab}:=g_{ab}-\hat n_a\hat n_b $
 & $\bar g_{ab}$ \\[2mm]
 $\nabla_a n_b$ &
 $\nabla^\top_a \hat n_b^{\phantom{top}} $
& $\II_{ab}$\\[2mm]
$\frac1{d-1} \, \nabla_a n^a$&
 $\frac 1{d-1}\,  \nabla^\top_a \hat n^a $
& $H$\\[2mm]
$\nabla_a n_b - \frac{(g_{ab}-n_a n_b) \nabla_c n^c}{d-1}$&
$\nabla^\top_a \hat n_b^{\phantom{top}}-\gamma_{ab} H $ & $\IIo_{ab}$\\
\hline
\end{tabular}
\end{center}
So for example, off the hypersurface the notation $\hat n_a$ means $(\nabla_{a}s)/|\nabla s|$ and an expression  like $\nabla_n H$ denotes $\frac1{d-1}\nabla_n \nabla_a^\top \hat n^a$. 
The label $\top$ is employed to denote projection 
of tensors on~$M$ 
  to their part perpendicular to the   unit vector $\hat n$, and $\nabla^\top:=\nabla - \hat n \, \nabla_{\hat n}$.

Our first key variational identity was given in Example~\ref{copre}, but we repeat it for completeness:
\begin{equation}\label{deltanhat}
 \delta  \hat{n}^a 
 =   \frac{\nabla^a \delta s}{ |n|} - \frac{n^a}{|n|^2} \nabla_{\hat{n}} \delta s  
\  \stackrel 1{=} \  \nabla^{\top}_a \delta s
\, \= \, \nablab_a\,  \sideline{\delta s}\, ,
\end{equation}
where $\sideline{\delta s}:=\delta s|_\Sigma\stackrel1=\hat \delta s|_\Sigma=:\sideline{\hat\delta  s}$. Thus 
\begin{equation}
\label{expunge}
\big(\delta \II_{ab}\big)^{\!\top}=\Big[
\nabla_a^\top\Big( \frac{\nabla_b^\top \delta s}{|n|}\Big)-
\frac{\nabla_a^\top \delta s}{|n|}\, 
\nabla_{\hat n} \hat n_b\Big]^\top
\stackrel1=\big(\nabla_a^\top\nabla_b^\top\delta s\big)^{\!\top}\, \= \, 
\nablab_a \nablab_b\,  \sideline{\delta s}\, .
\end{equation}
Importantly,  in the above computations, we first varied preinvariant expressions for invariants and only thereafter specialized these to a unit defining function.

In a functional setting, we can use
Equation~\nn{expunge} to give a useful lemma for computing variations of functionals involving the second fundamental form:

\begin{lemma}
Let $K^{ab}\in \Gamma(\odot^2 T\Sigma)$ and $s$ be a unit defining function with $\sideline{\delta  s}=\delta s|_\Sigma$ where~$\delta s$ has compact support. Then 
\begin{equation}\label{varid}\int_\Sigma dA_{\bar g}\,  K^{ab} \big(\delta \II_{ab}\big)\big|_\Sigma \stackrel1=  \int_\Sigma dA_{\bar g} \, \sideline{\delta s}\, \nablab_a \nablab_b K^{ab} \, .
\end{equation}
\end{lemma}

Finally, one more technical lemma is required:

\begin{lemma}\label{normalo}
Let $X_{ab}$ be a rank two, symmetric tensor on~$M$ satisfying 
$\hat n^a X_{ab}=0$. Then 
$$\big[\nabla_{\hat n} X_{ab} -\gamma_{ab} \, \gamma^{cd} \nabla_{\hat n}X_{cd} -\nabla_{\hat n} \big(X_{(ab)\circ}\big)\big]^\top=0\, .$$
\end{lemma}
\noindent We leave the proof to the reader.

\subsection{Examples} 

As a first example, we consider the problem of extremizing the area of an embedded hypersurface. The relevant functional is 
$$
I = \int_\Sigma dA_{\bar g} \, .
$$
Using Theorem~\ref{VARY}, we have
$$
\delta I = -\int_\Sigma dA_{\bar g} \ 
\measure\,
(\delta_{_{\text{-}\bar d}} \, 1)= - \, \bar d \ \cdot \int_\Sigma dA_{\bar g} \, 
\measure\  
H \, ,
$$
which recovers the standard vanishing  mean curvature~$H$ condition for minimal surfaces.

A second example is the variation of the 
Willmore 
energy  for surfaces (in Lorentzian signature, this is the rigid string action of~\cite{Polyakov}):
\begin{equation}\label{I2}
{\mathcal I}_2 =-\frac16 \int_\Sigma dA_{\bar g}\,  K \, ,
\end{equation}
where~$K$ is the rigidity density of Equation~\nn{rigidity}.  Using Theorem~\ref{vary_dA}, we find
$$
\delta {\mathcal I}_2 \stackrel 1= -\frac16 \int_\Sigma dA_{\bar g} \, (\delta K)|_\Sigma +\frac16\int_\Sigma dA_{\bar g} \, \sideline{\delta s} \, \delta_n K \, .
$$
Using Equation~\nn{varid}, and observing that only the variations of $\II_{ab}$ contribute when  the  preinvariant $K$ is written in terms of $\II_{ab}$ and $\gamma^{ab}$,
the first of the terms  above becomes
$$
-\frac16\int_\Sigma dA_{\bar g} \, (\delta K)|_\Sigma \stackrel 1= -\frac13\int_\Sigma dA_{\bar g}\, \delta s|\,  \nablab^a\nablab^b \IIo_{ab} \, .
$$
The second term requires that we compute a normal derivative of~$K$:
\begin{equation*}
\begin{split}
\delta_n K -2HK= \nabla_n K  &= 2\IIo^{ab}\nabla_n\IIo_{ab} \stackrel 1= 2\IIo^{ab}[\nabla_n,\nabla_a]n_b  \\[1mm]
&= 2\IIo^{ab}\big[- \Rho_{ab} - \IIo^2_{ab} - 2H\IIo_{ab}\big ]  \\[1mm]
&= - 2\tr(\IIo\Rho) - 4HK \, ,
\end{split}
\end{equation*}
where we have used 
Equation~\nn{Robin} as well as 
that the Weyl tensor vanishes for $d < 4$ and for surfaces~$\tr\IIo^3 = 0$. Assembling the above we find
$$
\delta {\mathcal I}_2 = -\frac13\int_\Sigma dA_{\bar g} \, \measure\, \big( \nablab_a\nablab_b + \Rho_{ab}^\top + H\IIo_{ab} \big ) \IIo^{ab}\, .
$$
Thus, using Equation~\nn{BGG}, we obtain the following:
\begin{proposition}\label{gradient2}
 With respect to hypersurface variations 
 the gradient of the Willmore energy ${\mathcal I}_2$ is
  $$-\frac13L^{ab}\, \IIo_{ab}\, .$$
\end{proposition}
\noindent
Thus the Willmore equation is $L^{ab}\IIo_{ab}=0$.
Note that the Willmore invariant displayed in the above Proposition
is manifestly conformally invariant because $L^{ab}$ is the invariant operator of Proposition~\ref{BGGop}.
 For Euclidean host spaces, using Equation~\nn{LapH}, the Willmore equation takes the more familiar form
$$
\bar \Delta H + 2H(H^2-{\rm K})=0\, ,
$$
where ${\rm K}=\frac12\,  \overline \Sc$ is the Gau\ss\  curvature.

\section{The conformal  four-manifold hypersurface energy functional}\label{cfm}

Our main variational result, as stated in Proposition~\ref{gradient}, 
is that  the obstruction density~${\mathcal B}_3$ for spaces agrees with  the  variational gradient of the {\it rigid membrane functional} 
$$
{\mathcal I}_3 =\frac16 \int_\Sigma dA_{\bar g}\, L \, ,
$$
where 
the membrane rigidity density~$L$ is given in Equation~\nn{mrd}.
In this section we prove Proposition~\ref{gradient}. 

We begin  by using Theorem~\ref{vary_dA} to write the variation  as
$$
\delta {\mathcal I}_3= \frac16\int_\Sigma dA_{\bar g}\, \big(\delta L\big)\big|_\Sigma - \frac16\int_\Sigma dA_{\bar g} \, 
\measure\, 
 \delta_{\hat{n}} L \, .
$$
To compute the first of these terms, we use the Fialkow Equation~\nn{Fialkow} to write
$$
L = \tr\IIo^3 - W(\hat n,\IIo,\hat n) \, ,
$$
so that
$$
\int_\Sigma dA_{\bar g} \big(\delta L\big)\big|_\Sigma = \int_\Sigma dA_{\bar g}\big[ 3\IIo_{ab}^2 \, \delta \IIo^{ab} - (\delta\IIo^{ab})\Wn_{ab} - 2\IIo^{ab}\Wn_{bac}\, \delta n^c \big]\big|_\Sigma \, .
$$
Using Equations~\nn{deltanhat} and~\nn{varid}, this becomes
\begin{equation*}
\begin{split}
\int_\Sigma dA_{\bar g} \, \delta L \stackrel1= \int_\Sigma dA_{\bar g}\, \sideline{\delta s}\,  &\big[ 3\nablab^a\nablab^b\IIo^2_{(ab)\circ} - \nablab^a\nablab^b \Wn_{ab} + 2\nablab^c(\IIo^{ab}\Wn_{bac}^\top) \big ]\big|_\Sigma \\[1mm]
\stackrel1= \int_\Sigma dA_{\bar g}\, \sideline{\delta s} \, &\big[ 3\nablab^a\nablab^b\IIo^2_{(ab)\circ} - \nablab^a\nablab^b \Wn_{ab} + \Wn{}_{bac}^\top \Wn^{bac} + 2\IIo^{ab}\nablab^c\Wn_{bac}^\top \big ]\big|_\Sigma \, ,
\end{split}
\end{equation*}
where  the final two terms in the last line were obtained using  Equation~\nn{tfcm} to rewrite the final term of the first line.

To compute the second term in~$\delta {\mathcal I}_3$, we employ Lemma~\ref{normalo} to   first calculate
the normal derivative of the second fundamental form:
\begin{equation*}
\begin{split}
(\nabla_n\IIo_{ab})^\top &\stackrel1= (\nabla_n \II_{ab})^\top_\circ = ([\nabla_n,\nabla_a]n_b)^\top_\circ = \Wn_{ab} - \Rho_{(ab)\circ}^\top - \II^2_{(ab)\circ} \\[1mm]
&= 2\Wn_{ab} - \bar \Rho_{(ab)\circ} - 2\IIo^2_{(ab)\circ} - H\IIo_{ab}\, ,
\end{split}
\end{equation*}
where in the final step we have used the Fialkow Equation~\nn{Fialkow}. Additionally, we must compute a normal derivative of the Weyl tensor:
\begin{equation*}
\begin{split}
\nabla_n \Wn_{ab} &\stackrel1= n^dn^c\nabla_c \Wn_{bad}  = (g^{dc} - \bar g^{dc})\nabla_c \Wn_{bad}  \\[1mm]
&= \II^{dc}W_{dabc} + n^c \nabla^d W_{dabc} - \nabla^\top_{c}\Wn_{ba}{}^{c} \\[1mm]
&= \IIo^{dc}W_{dabc} - 3H\Wn_{ab} + n^c C_{bca} - \nablab^c\Wn_{bac}^\top + \IIo^{c}_a \Wn\!^{\phantom{c}}_{cb} \, .
\end{split}
\end{equation*}
Orchestrating the above results we find
\begin{equation*}
\begin{split}
\delta_n L :=& \nabla_n L + 3HL \stackrel1= 3\IIo^2_{ab}\nabla_n \IIo^{ab} - (\nabla_n \IIo^{ab})\Wn_{ab} - \IIo^{ab}\nabla_n \Wn_{ab} \\[1mm]
=& \, 3\bar\Rho^{ab}\IIo^2_{(ab)\circ} - W(n,\bar \Rho,n) - HW(n,\IIo,n) - C(n,\IIo) - \nablab^c\IIo^{ab}\Wn_{bac}^\top \\[1mm]
&- 7W(n,\IIo^2,n) + K^2 + \IIo^{ab}\IIo^{cd} W_{dabc} + 2\Wn\!_{ab}\Wn^{ab} \, .
\end{split}
\end{equation*}
Using these, the variation~$\delta {\mathcal I}_3$ becomes
\begin{equation*}
\begin{split}
\delta {\mathcal I}_3 \stackrel1= \frac16\int_\Sigma dA_{\bar g} \, \sideline{\delta s} &\Big [ 3\nablab^a\nablab^b \IIo^2_{(ab)\circ} + 3\bar \Rho^{ab}\IIo^2_{(ab)\circ} - \nablab^a\nablab^b \Wn_{ab} - W(n,\bar \Rho, n) \\[1mm]
& + \IIo^{ab}\nablab^c \Wn_{bac}^\top - C(n,\IIo) - HW(n,\IIo,n) + \Wn_{bac}^\top \Wn^{bac} \\[1mm]
& - 7W(n,\IIo^2,n) + K^2 + \IIo^{ab}\IIo^{cd}W_{dabc} + 2\Wn_{ab}\Wn^{ab} \Big ]\Big|_\Sigma \, .
\end{split}
\end{equation*}
This result can be written in a manifestly conformally invariant way using the 
operator~$L^{ab}$
of Proposition~\ref{BGG} and the  hypersurface Bach tensor of Lemma~\ref{hsb}:
\begin{equation*}
\begin{split}
\delta {\mathcal I}_3 = \frac16\int_\Sigma dA_{\bar g} \, \measure\,  \Big ( L^{ab}  \big[3\IIo^2_{(ab)\circ}& - \Whn_{ab} \big] - \IIo^{ab}\Bach_{ab} + K^2+ \Whn_{bac}^\top \Whn^{bac} \\[1mm]
& - 7W(\hat n,\IIo^2,\hat n)  + \IIo^{ab}\IIo^{cd}W_{dabc} + 2\Whn_{ab}\Whn^{ab} \Big )\, .
\end{split}
\end{equation*}
This shows that the tensor
$$
\B_3 \!= \!\frac16\Big[ \!L^{ab\,}\! \!
\hspace{.2mm}
\big(3\IIo^2_{\!(ab)\!\hspace{.2mm}\circ} \hspace{.4mm}\!\!-\! \hspace{.2mm}\Whn_{ab}\big) -
\IIo^{ab}\!B_{ab}
 + K^2\! -7\Whn^{ab}\IIo^2_{ab} + 2\Whn_{ab}\Whn^{ab}\! + \IIo^{ab}\IIo^{cd}W_{\!cabd} + \Whn\!_{abc}^\top\Whn^{abc}\Big] 
$$
is the variational gradient and completes the proof of Proposition~\ref{gradient}.

\section{Conformal hypersurfaces and the singular Yamabe problem}\label{tractors}

We define here a  {\it singular Yamabe problem} as the boundary version of the classical Yamabe problem of finding a conformally related metric with constant scalar curvature~\cite{Aviles,MazzeoC,ACF}:

\begin{problem}\label{AYP}
Given a $d\geq 3$-dimensional Riemannian manifold~$(M,g)$ with boundary~$\Sigma:=\partial M$, find a smooth real-valued
function~$u$ on~$M$ satisfying the  conditions: 
\begin{enumerate}
\item~$u~$ is a defining function for~$\Sigma$ ({\it {\it i.e.}},~$\Sigma$
  is the zero set of~$u$, and~$\boldsymbol{d} u_x\neq 0$~$\forall x\in \Sigma$);
\item\label{(2)}~$\g:=u^{-2}g$ has scalar curvature
${\rm Sc}^{\g}=-d(d-1)$. 
\end{enumerate}
\end{problem}
Where~$0<u:=\rho^{-2/(d-2)}$, part (2) of this
problem is governed by the Yamabe equation
\begin{equation*}\label{Yamabe}\Big[\!-4\, \frac{d-1}{d-2}\, \Delta^g+\Sc^g\Big]\rho +  d\, (d-1)\,  \rho^{\frac{d+2}{d-2}}=0\, .\end{equation*}
But since~$u$ is a defining function, to deal with boundary aspects the equation 
 can be usefully recast: Since each $g\in \cc$ is in~1:1 correspondence with a true scale $\tau\in \Gamma(\ce M[1])$, setting the smooth defining function $u=\sigma/\tau$, the Yamabe equation becomes
\begin{equation}\label{Ytwo}
S(\sigma):= 
\big(\nabla \si \big)^2 - \frac{2}{d} \, \si\, \Big(\Delta +\frac{\rm Sc}{2(d-1)}\Big) \si = 1 \, . 
\end{equation}
In the above $\Delta:=\bg^{ab}\nabla_a \nabla_b$ and all index contractions are performed with the conformal metric $\bg$ and its inverse.
It follows easily that  the quantity $S(\sigma)$ on the left hand side of Equation~\nn{Ytwo} is conformally invariant. 
Note that $S(\sigma)=-\frac{Sc^{\g}}{d(d-1)}$.
Also, because $u$ is a defining function, $\sigma$ is a {\it defining density}: this is a section of~$\Gamma(\ce M[1])$ with zero locus ${\mathcal Z}(\sigma)=\Sigma$ and  such that for any Levi-Civita connection in the conformal class $\nabla \sigma\neq 0$ along $\Sigma$; see Section~\ref{geom-Sc} below.

Development of the asympototic analysis of the singular Yamabe problem as well as uncovering its fundamental role in the theory of conformal hypersurface invariants  uses the tractor calculus treatment of conformal geometries. 
We next review key aspects of that theory.

\subsection{Conformal geometry and tractor calculus} \label{ASC-sec}

Although there is
no distinguished connection on the tangent bundle
$TM$ for
a $d$-dimensional conformal manifold~$(M,\cc)$,
 there is a canonical metric~$h$ and linear connection
$\nabla^{\ct}$ (preserving~$h$) on a related higher rank vector bundle known as
the tractor bundle~$\ct M$. 
This  is a
rank~$d+2$ vector bundle equipped with a canonical
{\it tractor connection}~$\nabla^\ct$~\cite{Thomas,BEG}. The bundle~$\ct M$ is not
irreducible but has a  composition series given by the   
 semi-direct sum 
  $$ \cT^AM= \ce M[1]\lpl
\ce_a M[1]\lpl\ce M[-1]\, .$$
There is  a canonical  bundle
inclusion~$ \ce M[-1] \to \cT^A M$, with~$T^*M[1]$ a subbundle of the
quotient by this, as well as a surjective bundle map~$\cT M\to
\ce M[1]$.  We denote by~$X^A$ the canonical section of~$
\cT^A M[1]:=\cT^A M\otimes \ce M[1]$ giving the first of these:
\begin{equation}\label{X-incl}
X^A:\ce M[-1]\to \cT^A M\,  ,
\end{equation}
and term~$X$  the {\em canonical tractor}.

A choice of metric~$g \in \cc$ determines an
isomorphism (or {\it splitting})
\begin{equation}\label{split}
\mathcal{T} M\stackrel{g}{\cong} \ce M[1]\oplus T^*\!M[1]\oplus
\ce M[-1] ~.
\end{equation}
We may write, for example,~$U\stackrel{g}{=}(\si,~\mu_a,~\rho)$, or
alternatively~$[U^A]_g=(\si,~\mu_a,~\rho)$, to mean that~$U$ is an invariant
section of~$\ct M$ and~$(\si,~\mu_a,~\rho)~$ is its image under the
isomorphism~\nn{split}; often the dependence on  $g\in \cc$ is supressed when this is clear by context.
Changing to a conformally related metric
$\widehat{g}=\Omega^2g$  gives a different
isomorphism, which is related to the previous one by the transformation
formula
\begin{equation*}\label{transf}
{(\si,\mu_a,\rho)}^{\sss\Omega^2 g}=\big(\Omega\si,\, \Omega(\mu_a+\si\Up_a),\, \Omega^{-1}(\rho-g^{bc}\Up_b\mu_c-
\tfrac{1}{2}\si g^{bc}\Up_b\Up_c)\big)^g,  
\end{equation*}
where~$\Upsilon_b$ is the one-form~$\Omega^{-1}\dd\Omega$.

 In terms of the above {splitting}, the tractor connection is given by 
\begin{equation}\label{trconn}
\nd^{\ct}_a
\left( \begin{array}{c}
\si\\[1mm]\mu_b\\[1mm] \rho
\end{array} \right) : =
\left( \begin{array}{c}
    \nabla_a \si-\mu_a \\[1mm]
    \nabla_a \mu_b+ \bg_{ab} \rho +\Rho_{ab}\si \\[1mm]
    \nabla_a \rho - \Rho_{ac}\mu^c  \end{array} \right) .
\end{equation} 
We will often recycle the Levi-Civita notation
to also denote  the tractor connection; it is the only connection
we shall use on~$\ct M$, its dual, and tensor powers. Its curvature is  the tractor-endomorphism valued two form
$
{\mathcal R}^{\bm \sharp} 
$;
in the above splitting this acts as
\begin{equation}\label{curvature}
{\mathcal R}_{ab}^{\bm \sharp}\left( \begin{array}{c}
\si\\[1mm]\mu_c\\[1mm] \rho
\end{array} \right) 
=
[\nabla_a,\nabla_b] 
\left( \begin{array}{c}
\si\\[1mm]\mu_c\\[1mm] \rho
\end{array} \right)  =
\left( \begin{array}{c}
   0 \\[1mm]
    W_{abc}{}^d \mu_d+ C_{abc} \sigma  \\[1.3mm]
    -C_{abc} \mu^c
    \end{array}\right)\, .
\end{equation}

 For~$[U^A]=(\si, \mu^a,
\rho)$ and~$[V^A]=(\tau, \nu^a,\kappa)$,  the conformally invariant, signature $(d+1,1)$,  {\em tractor metric}~$h$ on~$\mathcal{T}M$  given by
\begin{equation}\label{trmet}h(U,V)=h_{AB}U^A V^B=\si \kappa +\bg_{ab}\mu^a\nu^b+\rho\tau=:U\cdot V \, ,\end{equation}
is preserved by the tractor connection, {\it i.e.},~$\nabla^\cT h=0$. 
It follows from this formula that~$X_A=h_{AB}X^B$ provides the surjection by contraction
$\iota(X):\cT M\to \ce[1]$.
   The tractor metric~$h_{AB}$
and its inverse~$h^{AB}$ are used to identify~$\ct M$ with its dual  and to raise and lower tractor
indices. The  notation~$V^2$ is shorthand for~$V_A V^A=h(V,V)$.

Tensor powers of the standard tractor bundle~$\ct M$, and tensor products
thereof, are vector bundles that are also termed tractor bundles. We
shall denote a  tractor bundle of arbitrary tensor type by~$\cT^\Phi M$ and write
$\cT^\Phi M[w]$ to mean~$\cT^\Phi M\otimes \ce M[w]$;~$w$ is then said to be
the weight of~$\cT^\Phi M[w]$. 

Closely linked to~$\nabla^\ct$ is an important second order
conformally invariant differential operator $$ D^A : \Gamma (\cT^\Phi M
[w])\to \Gamma (\cT^{\Phi'}\!M[w-1])\, ,$$   known as the Thomas
(or tractor) D-operator. Here  $\cT^{\Phi'}\!M[w-1]:=\cT^A\otimes \cT^\Phi M[w-1]$. In a scale~$g$, 
\begin{equation}\label{Dform}
[D^A  ]_g =\left(\begin{array}{c} (d+2w-2)\, w\  \hspace{.3mm} \\[1mm]
(d+2 w-2) \nabla_a \\[1mm]
-(\Delta^g+ \J w)    \end{array} \right) ,
\end{equation}
where here~$\Delta=\bg^{ab}\nabla_a\nabla_b$, and~$\nabla$ is the coupled
Levi-Civita-tractor connection~\cite{BEG,Thomas}. 
The following variant of the Thomas D-operator is also useful.
\begin{definition}\label{hD}
Suppose that~$w\neq1-\frac d2$. The operator
$$
\hD^A:\Gamma(\cT^\Phi M[w])\longrightarrow \Gamma\big(\ct^{\Phi'} M[w-1]\big)
$$
is defined by
$$
\hD^A T :=  \frac1{d+2w-2} \, D^A T\, .
$$
\end{definition}

\begin{remark}
We term the critical value $w=1-\frac d2$ the {\it Yamabe weight}, since one then has
$$
D^A=-X^A \square_{Y}\, ,
$$
where the operator $-\square_{Y}:\Gamma(\ct^\Phi M[1-\frac d2])\to
\Gamma(\ct^\Phi M[-1-\frac d2])$ is the conformally invariant, tractor-coupled,  Yamabe operator given by $\Delta -\frac{d-2}{2}\, \J$ in a choice of scale.
\end{remark}

The Thomas D-operator is null, in the sense that
$$
D^A \circ D_A = 0\, ;
$$
this operator identity can easily be derived from Equation~\nn{Dform}. The Thomas D-operator is not, however, a  derivation. Its failure to obey a Leibniz rule is captured by the following Proposition, a result which was first observed in~\cite{Joung} and proved in~\cite{GW15}:

\begin{proposition}\label{leib-fail}
Let~$T_i\in \Gamma(\cT^\Phi M[w_{i}])$ for~$i=1,2$, and~$h_{i}:=d+2w_{i}$,~$h_{12}:=d+2w_1+2w_2-2$ with~$h_i\neq0\neq h_{12}$.
Then
\begin{equation}\label{ls}
\hD^A(T_1T_2)- (\hD^A T_1) \, T_2 - T_1(\hD^A T_2)=-\frac{2}{d+2w_1 + 2w_2 -2}\, X^A\, (\hD_B T_1)(\hD^B T_2)\, .
\end{equation}
\end{proposition}

We shall also need the following Lemma which is easily verified by direct application of Equation~\nn{Dform} or the modified Leibniz rule~\nn{ls} in tandem with the identity
\begin{equation}\label{DX}\hD^A X^B=h^{AB}\, .\end{equation}
\begin{lemma}
Let $T\in \Gamma(\ct^\Phi M[w])$. Then 
\begin{equation}\label{DXT}
D_A(X^AT)=(d+w)(d+2w+2)T\, .
\end{equation}
\end{lemma}

\subsection{Conformal hypersurfaces and tractors}\label{htractors}

The basis for a conformal hypersurface calculus~\cite{Goal,Grant,Stafford,YuriThesis,GW15}
is a unit tractor
object~$N^A\in \Gamma(\ct M)|_\Sigma$ (from~\cite{BEG}) corresponding to the conformal unit conormal~$\hat n$, termed the {\em normal tractor}.
This  is defined along~$\Sigma$, in a choice of scale,  by
\begin{equation}\label{normaltractor}
[N^A]= \begin{pmatrix}
0\\\, \hat n^a\\ -H
\end{pmatrix}\Rightarrow N_A N^A =1\, .
\end{equation}
The subbundle~$\ct M^\top $ orthogonal to the normal tractor (with respect to the tractor metric~$h$)
along~$\Sigma$
is canonically isomorphic to the intrinsic hypersurface tractor bundle~$\ct\Sigma$, see~\cite{BrGoCNV} and~\cite[Section 4.1]{Goal}.  This is  the
conformal tractor analog
of the Riemannian isomorphism between the intrinsic tangent bundle~$T\Sigma$ and the subbundle~$TM^\top\!$ of
~$TM|_\Sigma$ orthogonal to~$\hat n^a$ along~$\Sigma$.
This isomorphism identifies these bundles and   thus we use the same abstract index for~$\ct M$ and~$\ct \Sigma$. In explicit computations in a given choice of scale,  we  will  need  to
relate sections given in corresponding splittings of the ambient and hypersurface tractor bundles.
 In terms of sections expressed in a scale~$g\in \cc$ (which determines $\bar g\in\bar\cc$), the isomorphism is given by 
\begin{equation}\label{Tisomorphism}
\big[V^A\big]_g:=\begin{pmatrix}v^+\\[1mm]v_a\, \\[2mm]v^-\end{pmatrix}\mapsto
\begin{pmatrix}
v^+\\[1mm]v_a-\hat n_aH v^+\\[2mm]v^-+\frac12 H^2 v^+
\end{pmatrix}=\big[U^A{}_B\big]_{\bar g}^g\, \big[ V^B\big]_g=:\big[\bar V^A\big]_{\bar g}\, ,
\end{equation}
where~$V^A\in\Gamma (\ct M^\top)$, $\bar V^A\in  \Gamma(\ct \Sigma)$ and
 the~$SO(d+1,1)$-valued matrix
$$
\big[U^A{}_B\big]_{\bar g}^g:=\begin{pmatrix}1&0&\ 0\ \\[2mm]-\hat n_a H&\delta_a^b&0\\[3mm]-\frac12 H^2&\hat n^b H&1\end{pmatrix}\, .
$$
In the above
 the canonical isomorphism between~$T^*M^\top\big|_\Sigma$ and~$T^*\Sigma$ defined by the unit  normal vector~$\hat n$ is used
to identify sections of these bundles. 

The ambient tractor connection~$\nabla$ obeys an analog of the classical Gau\ss\ formula relating it to the intrinsic tractor connection $\nablab$ of $(\Sigma,\cc)$ which is
known as the Fialkow--Gau\ss\ formula (\cite{Grant, Stafford, YuriThesis} or in our current notation~\cite[Proposition 7.14]{GW15}) because the Fialkow tensor of Equation~\nn{Fialkow} encodes the difference between these two connections along~$\Sigma$. As in the Gau\ss\ case, the tractor-coupled operator $\nabla^\top:=\nabla -\hat n \nabla_{\hat n}$ is well-defined along~$\Sigma$ irrespective of how tensors are extended to the ambient.
For future use, we record some identities for this operator acting on the normal tractor:
\begin{lemma}
Along~$\Sigma$, tangential, tractor-coupled gradients of the normal tractor are given in a choice of scale by
\begin{equation}
\label{nablaN}
\nabla^\top_a N^C =
\begin{pmatrix}
0\\ \IIo_{ac} \\[1mm] -\frac{\nablab.\IIo_a}{\dbs\!-1}\ 
\end{pmatrix}\, ,\quad
\nablat_a\nablat_b N^C = \left( \begin{array}{c} -\IIo_{ab} \\[1mm] \nablat_a\IIo_{bc} - \frac{1}{\dbs\!-1}\bar g_{ac}\nablab.\IIo_b \\[1mm]
- \frac1{\dbs\!-1}\nablat_a\nablab.\IIo_b  -\IIo_b^c \Rho^\top_{ac} \end{array} \right) \, ,
\end{equation}
while the tangential, tractor-coupled operator $\Delta^{\!\top}:=g^{ab}\nabla^\top_a \nabla^\top_b$ gives
\begin{equation}\label{boxN}\Delta^{\!\top} N^C =\begin{pmatrix}
0\\[1mm] \frac{\dbs\!-2}{\dbs\!-1}\nablab.\IIo_c-n_c K
\\[2mm]
-\frac{\nablab.\nablab.\IIo+(\dbs\!-1)\Rho_{ab}\IIo^{ab}}{\dbs\!-1}
\end{pmatrix}\, .\end{equation}
\end{lemma}

\begin{proof}
Since $\nabla^\top$ is tangential, the first identity follows directly from~\cite[Corollary 6.7]{GW15}. The third identity follows directly from the second, both of which   only require a direct application of Equation~\nn{trconn}. \end{proof}

\begin{remark}
Note that  expressions such as $\nabla^\top_{a} \IIo_{bc}$ in the above Lemma are well-defined along~$\Sigma$ since $\nabla^\top$ is tangential. It is obviously possible to further develop the second and third identities  
using the Gau\ss\ formula~\nn{II} and (for $\bar d\geq 3$) the Fialkow--Gau\ss\ Equation~\nn{Fialkow}. 
\end{remark}

\subsection{Singular Yamabe problem asymptotics and defining densities}\label{geom-Sc}

Given a hypersurface~$\Sigma$, a section~$\sigma\in\Gamma(\ce[1])$ 
is said to be a {\it defining density} for~$\Sigma$ if~$\Sigma=\Z(\sigma)$ and~$\nabla \sigma$ is nowhere vanishing along~$\Sigma$, where $\nabla$ is the Levi-Civita connection for some $g\in\cc$.
The conformal analog of the 
exterior derivative of a defining function is then provided by the tractor field
\begin{equation}\label{sctrac-def}
I^A_\si:=\hD^A \si
\stackrel{g}{=} (\sigma,\nabla_a \si,-\frac{1}{d}(\Delta +\J)\si)=:(\sigma,n_a,\rho) \, .
\end{equation}
 In  Riemannian signature
 for any defining density~$\sigma$, we have that
\begin{equation*}\label{ge0}
I_\sigma^2>0
\end{equation*}
holds in a neighbourhood of~$\Sigma$. More generally, any section $\sigma\in \Gamma(\ce M[1])$ such that~$I^A_\sigma$ is nowhere vanishing is called a {\it scale}, $I^A_\sigma$ is its {\it scale tractor} and the data $(M,\cc,\Sigma)$ is an {\it almost Riemannian structure} (or geometry). The following proposition, whose proof follows directly from  Equations~\nn{sctrac-def} and~\nn{trmet}, shows that almost Riemannian geometries are a natural setting for the singular Yamabe problem:
\begin{proposition}[see~\cite{Goal}]\label{I2-prop}
For~$\si\in \Gamma(\ce[1])$ the quantity~$S(\si)$  of Equation~\nn{Ytwo} is the squared length of the corresponding scale tractor: 
\begin{equation}\label{Isq}
S(\si)=I^2_\si:= h_{AB}I^A_\si I^B_\si.
\end{equation}
\end{proposition}
Indeed, the Yamabe equation~\nn{Ytwo}
 now reads simply
$$
I_\sigma^2=1\, .
$$
Thus, the asymptotic version of the singular Yamabe Problem~\ref{AYP}
is stated as follows:
\begin{problem}\label{I2-prob} 
Find a smooth defining density~$\bar \sigma$ such that
\begin{equation}\label{ind}
I^2_{\bar\si}=1 + \bar{\sigma}^{\ell} A_\ell\, ,
\end{equation}
for some smooth~$A_\ell\in \Gamma(\ce[-\ell])$,
where~$\ell \in\mathbb{N}\cup\infty$ is as high as possible.
\end{problem}

The above problem was solved in~\cite{CRMouncementCRM,GW15}, that result  is captured by the following Theorem whose proof may be found in~\cite{GW15}:
\begin{theorem}\label{bigformaggio}
Let $\sigma$ be a defining density for a hypersurface~$\Sigma$ and $\sigma_0:=\sigma/\sqrt{I^2_\sigma}$, then there exist smooth densities $A_k\in\Gamma(\ce \Sigma[-k])$ such that
\begin{equation}\label{sigmabar}
\bar \sigma =\sigma_0\, 
\Big(1 + A_1 \sigma_0 + A_2 \sigma_0^2 + \cdots+ A_{d}\sigma_0^d
 \Big)
\end{equation} 
solves 
$$
I_{\bar \sigma}^2=1+\sigma_0^d B_\sigma\, ,
$$
for some smooth $B_\sigma\in \Gamma(\ce M[-d])$.
For $1\leq k\leq d$ the densities $A_k$ are determined by the recursion
$$\bar\sigma_k=
\bar\sigma_{k-1}\left[1-\frac d2\frac{I_{\bar\sigma_{k-1}}^2-1}{(d-k)(k+1)}\right]\, ,
$$
where $\bar\sigma_{k-1}:=\sigma_0(1+\sum_{j=1}^{k-1}\sigma^j
A_j)$ solves $I^2_{\bar\sigma_{k-1}}=1+\sigma^{k-1} C$ for smooth $C\in \Gamma(\ce M[1-k])$.
Moreover, for any true scale $\tau\in \Gamma(\ce M[1])$,
$$
\sigma^\prime=\bar\sigma
\left[1+\frac d2 \log(\bar \sigma/\tau)\, \frac{I_{\bar\sigma}^2-1}{d+1}\right]
$$
solves
$$
I_{\sigma^\prime}^2=1+\sigma^{d+1}\big(C^\prime+D\log(\sigma/\tau)+D^\prime \log^2(\sigma/\tau)\big)
$$
for smooth $C^\prime,D,D^{\prime}\in \Gamma(\ce M[-d-1])$. Finally, given two defining functions  $\sigma$ and $\tilde \sigma$ for~$\Sigma$, then 
$$
B_\sigma\big|_\Sigma=B_{\tilde\sigma}\big|_\Sigma=:\B_d\, .
$$
\end{theorem}

The last statement of this theorem implies that the  density $\B_d$ is uniquely determined by $(M,\cc,\Sigma)$. Since it obstructs a smooth all orders solution to the singular Yamabe problem  we term it the {\it obstruction density}.
The solution $\bar \sigma(\sigma)$  is determined uniquely by the recursion up to the order of the obstruction. Thus, up to the given order,~$\bar \sigma$ is a canonical scale  defining the hypersurface~$\Sigma$ and 
whose scale tractor has unit length to accuracy $\sigma^d$,
 so is termed a {\it conformal unit defining density}.

\subsection{Conformal hypersurface invariants and canonical extensions}\label{cinv}

A Riemannian hypersurface invariant $\Pre(\Sigma;g):=\Pre(s;g)|_\Sigma$ with the property $\Pre(\Omega s,\Omega^2 g)$ $=$ $\Omega^w$ $\Pre(s;g)$ is termed a {\it conformal hypersurface covariant}. 
This determines an invariant section $\Pre(\sigma,\bg)\in \Gamma(\ce M[w])$ where $\sigma$ is a defining density.
Thus $\Pre(\sigma,\bg)|_\Sigma$ is called a {\it conformal hypersurface invariant}. 
As for the Riemannian case, this definition extends to tensor and therefore also tractor-valued invariants. 

Since the conformal unit scale $\bar \sigma$ is a unique defining density,  
up to the addition of  terms~$\bar \sigma^{d+1} S$ for $S\in\Gamma(\ce M[-d])$, it follows that any conformal density $\Pre(\bar \sigma(\sigma),\bg)$, where $\bar \sigma(\sigma)$ is determined in terms a given defining density $\sigma$ via Equation~\nn{sigmabar} of Theorem~\ref{bigformaggio}, gives a conformal hypersurface invariant $\Pre(\bar \sigma(\sigma),{\bm g})|_\Sigma$ so long as the formula for $\Pre$ involves jets of $\bar \sigma$ of order~$\leq d$. 
(In fact, $\Pre(\bar \sigma(\sigma),{\bm g})$ is a density-valued preinvariant.)
This construction  provides {\it holographic formulae} for  conformal hypersurface invariants, see~\cite[Section 6.2]{GW15}.

\begin{example}
The scale tractor for the conformal unit defining density gives a holographic formula for the normal tractor 
$$
I_{\bar\sigma}^A|_\Sigma=\hD^A \bar \sigma|_\Sigma=N^A\, .
$$
This result was first proved in~\cite{Goal}.
\end{example}

Moreover, in the above example we may view the conformal unit defining density as providing a canonical extension $I^A_{\bar \sigma}$ of the normal tractor~$N^A$. A short computation (in a choice of scale) shows that  
$$
\nabla_a I^B_{\bar\si} =\begin{pmatrix}0\\ \nabla_a n_b + \Rho_{ab} \bar\sigma + g_{ab} \rho\\\star \end{pmatrix}\, ,
$$
where $\rho:=-\frac1{d}\, (\Delta \bar \sigma + \J \bar \sigma)$ as in Equation~\nn{sctrac-def}. Comparing this with Equation~\nn{nablaN} gives a holographic formula for the trace-free second fundamental form
$$
\IIo_{ab}=(\nabla_a n_b + \Rho_{ab} \bar\sigma + g_{ab} \rho)|_\Sigma\, .
$$
Hence we have the following canonical extensions of the trace-free second fundamental form and rigidity density
\begin{equation}\label{canexts}
\IIo_{ab}:=\nabla_a n_b + \Rho_{ab} \bar\sigma + g_{ab} \rho\, \mbox{ and }\, 
K:=(\nabla_a n_b + \Rho_{ab} \bar\sigma + g_{ab} \rho)(\nabla^a n^b + \Rho^{ab} \bar\sigma + g^{ab} \rho)\, .
\end{equation}
The above formul\ae\ are now   well-defined 
in $M$ rather than only along~$\Sigma$. These extensions are canonical up to a finite order determined by Theorem~\ref{bigformaggio}.

\subsection{The Laplace--Robin operator and tangential operators}\label{LRO}

By contracting the scale tractor and Thomas D-operator,
we can build a new operator that plays two roles: that of the  ambient Laplace operator and of  a conformally invariant boundary Robin-type Dirichlet-plus-Neumann operator~\cite{GoSigma}.
Thus we call the operator
$$
I\cdot D:\Gamma(\ct^\Phi M[w])\longrightarrow \Gamma(\ct^\Phi M[w-1])\, ,
$$
the {\it Laplace--Robin} operator. Calculated in some scale $g\in\cc$
\begin{equation}\label{IdotD}
\begin{split}
I\cdot D=
-\sigma \Delta + (d+2w-2)\big[\nabla_n-\frac wd(\Delta \sigma)\big]-\frac{2w}{d}(d+w-1)\sigma\J\, ,
\end{split}
\end{equation}
where $n=\nabla\sigma$. Hence this is a Laplace-type operator which is  degenerate  along $\Sigma=\Z(\sigma)$. 
Indeed, along $\Sigma$ the Laplace--Robin operator becomes first order. In particular, for a conformal unit defining scale,
along~$\Sigma$ and  in a choice of scale we have
$$
I_{\bar \sigma}\cdot D=
(d+2w-2)\big(\nabla_n-w H\big)\, .
$$
Hence, for $w\neq 1-\frac d2$ and in the spirit of the Section~\ref{cinv}, along~$\Sigma$ the Laplace-Robin operator is a operator-valued holographic formula for the Robin operator $\delta_{\hat n}$ as given   in  Equation~\nn{Robin} (recall that for a conformal unit defining density, along $\Sigma$ one has $\nabla \sigma=\hat n$). 
Given a natural density, tensor or tractor field~$\bar f$ along~$\Sigma$, it is often the case that we can fix an extension~$f$ off~$\Sigma$, that is canonical  to some order
using the conformal unit defining density, such as those given in Equation~\nn{canexts}). In that case we will denote the Robin operator by
\begin{equation*}\label{deltaR}
\delta_{\rm R}\bar f=\left\{
\begin{array}{cl}I_{\bar\sigma}\cdot\hD f\big|_\Sigma\, ,&w\neq 1-\frac d2\, ,\\[3mm]
\nabla_{\hat n} f\big|_\Sigma+\frac{d-2}2\, H\bar f\, ,&w=1-\frac d2\, .
\end{array}
\right.
\end{equation*}
In light of this, we will often use the same notation for $f$ and $\bar f$.

 Combining the conformal unit defining density and the Laplace--Robin operator allows successive normal derivatives of canonically extended conformal hypersurface invariants to 
 be defined  up to the order of the obstruction. 
 The canonical extensions of the normal tractor and intrinsic canonical tractor are the scale tractor $I^A_{\bar \sigma}$ and the ambient canonical tractor~$X^A$, respectively.
 The following Lemma gives one normal derivative acting on these:
 \begin{lemma}
 The Robin operator acting on the canonically extended canonical and normal tractors gives:
 \begin{equation}\label{deltaRX}\delta_{\rm R} X^A=N^A\, ,\qquad 
 \delta_{\rm R} N^A=\frac{\ K X^A\, }{\db-1}\, \, .
 \end{equation}
\end{lemma}
 
\begin{proof}
The first identity is a direct consequence of Equation~\nn{DX}.
For the second we compute in a scale
$$
\delta_{\rm R} N^A=\nabla_n I^A_{\bar \sigma}\big|_\Sigma=\begin{pmatrix}
0\\[1mm]
\nabla_n n_a -H n_a\\[1mm] \nabla_n\rho -\Rho(n,n)
\end{pmatrix}\, .
$$
To complete the proof we need  the normal derivatives of the canonical extensions of  $n_a$ and $\rho$  determined by $I^A_{\bar\sigma}$.
These were computed in~\cite[Lemmas 6.1 and 6.6]{GW15} and, along~$\Sigma$  are:
\begin{equation}\label{nablann}\nabla_n n_a = H n_a\, ,\qquad
\nabla_n \rho=\Rho(n,n)+\frac{K}{\db-1}\, .
  \end{equation} 
\end{proof}

The Laplace--Robin  operator has one further crucial role to play. In~\cite{GW}, it was shown that by viewing the scale~$\sigma$ as an operator mapping
$\Gamma(\ct^\Phi M[w])\to\Gamma(\ct^\Phi M[w+1])$ and acting by multiplication, then the commutator of this with the Laplace--Robin operator acting on weight~$w$ tractors obeys
$$
[I_\sigma\cdot D,\sigma]=-I_\sigma^2 \, (d+2w)\, .
$$
This identity underlies a solution-generating ${\frak sl}(2)$ algebra. A particular consequence of this is that the operator
$$
{\mathcal P}_{k}^\sigma:=\Big(-\frac1{I_\sigma^2}\, I_\sigma\cdot D\Big)^k
$$
is tangential when acting on weight $w=\frac{k-d+1}2$ tractors. Specializing to the conformal unit defining density gives a differential operator
\begin{equation*}
\label{Pk}
{\sf P}_k:\Gamma\big(\ct^\Phi M\big[\frac{k-d+1}{2}\big]\big)\Big|_\Sigma\longrightarrow \Gamma\big(\ct^\Phi M\big[\frac{-k-d+1}{2}\big]\big)\Big|_\Sigma\, ,\end{equation*}
which is completely determined by the data $(M,\cc,\Sigma)$~\cite[Section 8.1]{GW15}.
We call this the 
{\it extrinsic conformal Laplace operator}. For $k$ even the operator ${\sf P}_k$ is $(\Delta^{\!\top})^{\frac k2}+{\rm L.O.T.}$, up to multiplication by a non-vanishing constant, where ``L.O.T.'' stands for terms of lower derivative order involving both intrinsic and extrinsic curvature quantities.

\section{Singular Yamabe obstruction densities}\label{spaces}

Given a unit conformal defining density, the obstruction density~$\B_{\bar d}$ has a simple holographic formula whose leading structure is given in terms of the extrinsic conformal Laplace operator~\cite[Theorem 8.11]{GW15}:

\begin{theorem}
The obstruction density
is given by 
\begin{equation}\label{holo-form}
  \B_{\bar d}= \frac{2}{\db!(\db+1)!}\, \bar D_A \Big[\Sigma^A_B\Big({\sf P}_{\dbs} \! N^B + (-1)^{\dbs\!-1} \big[\bar I\cdot D^{\dbs\!-1}(X^B K)\big]\big|_\Sigma\Big)\Big]\, ,
\end{equation}
where the rigidity density $K$ is canonically extended off~$\Sigma$ by the formula~$K=P_{AB}P^{AB}$ with~$P^{AB}:=\hD^A\bar I^B$. Also, the projector $\Sigma^A_B:=\delta^A_B-N^A N_B$.
\end{theorem}

The main aim of this section is to compute ~${\mathcal B}_3$ using the above holographic formula.
The canonical extension of the rigidity density~$K$ stated in the theorem is the same as that given in Equation~\nn{canexts}.
Also, in hypersurface dimensions $\db=2,3$, the extrinsic conformal Laplacians acting on weight zero tractors are given by
\begin{equation}\label{P23}
\begin{split}
{\sf P}_2\ &=\ \ \Delta^{\!\top}\, ,\\[1mm]
{\sf P}_3\ &=-8\, 
\IIo^{ab}\nabla_a^\top\nabla_b^\top-8\, \nablab.\IIo^b\nabla_b^\top+4n^a\Big( {\mathcal R}^{\bm \sharp}_{ab}\circ\nabla^b
+ \, \nabla^b \circ {\mathcal R}_{ab}^{\bm \sharp}\Big)\, .
\end{split}
\end{equation}

\begin{proposition}\label{P2N}
If $\db=2$, then
\begin{eqnarray*}
{\sf P}_2 N^C&=&-
X^C L_{ab}\IIo^{ab}- N^C K\, ,\quad \db=2\, .
\end{eqnarray*}
\end{proposition}
\begin{proof}
Using Equations~\nn{P23} and~\nn{boxN}, we have for~$\db = 2$,
$$
{\sf P}_2N^C = \Delta^{\!\top} N^C = \begin{pmatrix} 0 \\ -n^c K \\ -\nablab.\nablab.\IIo - \Rho_{ab}\IIo^{ab} \end{pmatrix} \, .
$$
Using Equation~\nn{BGG} the quoted result follows.
\end{proof}
\begin{remark}
Our computation of ${\mathcal B}_3$ relies  on a~$\db = 3$ version of the above Proposition.
\end{remark}

The second term in Equation~\eqref{holo-form} involves normal derivatives of the canonically extended rigidity density~$K$. The following proposition provides the first of these.

\begin{proposition}\label{deltaRK}
The canonically extended trace-free second fundamental form~$\IIo$, rigidity density~$K$ and membrane rigidity density~$L$ obey 
\begin{equation}\label{deltaRK}
\begin{split}
\big(\delta_R \IIo_{ab}\big)^{\!\top} &= -\IIo^2_{ab} + \Whn_{ab} + \frac12 \bar g_{ab} K \, , \\[1mm]
\delta_R K\ \  &= -2(\db-2)L \, .
\end{split}
\end{equation}
\end{proposition}

\begin{proof}
We begin by computing a normal derivative of the canonical extension of~$\IIo_{ab}$:
\begin{equation}\label{nablanIIo}
\begin{split}
\nabla_n \IIo_{ab} =& \nabla_n\nabla_a n_b + g_{ab}\nablan \rho + \sigma \nablan\Rho_{ab} + \Rho_{ab}\nablan\sigma \\[1mm]
=& \Rn_{ab} - (\nabla_an_c)\nabla^cn_b - \nabla_a\nabla_b(\rho\sigma) + g_{ab}\nablan\rho + \Rho_{ab} + \sigma\big( \nablan-2\rho\big) \Rho_{ab} \\[1mm]
=&\Rn_{ab}  -\IIo^2_{ab}  + \rho\IIo_{ab} - 2n_{(a}\nabla_{b)}\rho + g_{ab}\nablan\rho + \Rho_{ab} \\[1mm]
&+\sigma\Big[\big(\nablan - 3\rho\big) \Rho_{ab} + 2\IIo_{c(a}\Rho_{b)}^c - \nabla_a\nabla_b\rho\Big] +  \mathcal{O}(\sigma^2) \, .
\end{split}
\end{equation}
Projecting onto  the tangential piece of this quantity along~$\Sigma$, and using Equation~\nn{nablann} we find
\begin{equation*}
\big(\nabla_n\IIo_{ab}\big)^{\!\top} \stackrel\Sigma= -\IIo^2_{ab} + \Wn_{ab} + \frac12\bar g_{ab} K - H\IIo_{ab} \, .
\end{equation*}
Using Equation~\nn{deltaRab}, we have
\begin{equation*}
\big( \delta_{\rm R}\IIo_{ab}\big)^{\!\top} \stackrel \Sigma= \big[ (\nablan + H) \IIo_{ab}\big]^\top \stackrel\Sigma = -\IIo^2_{ab} + \Wn_{ab} + \frac12 \bar g_{ab} K \, .
\end{equation*}
The second equation follows by noting that~$\delta_R K = 2\IIo^{ab}\delta_R\IIo_{ab}$ and using the above result.
\end{proof}

\begin{remark}
Recalling that the Willmore energy functional~\nn{I2} for embedded surfaces is the integral of the rigidity density~$K$ and its analog~\nn{I3} for embedded spaces is the integral of the membrane rigidity density~$L$, we see from Equation~\ref{deltaRK} that the Robin operator relates the two integrands for these functionals. It would be interesting to investigate whether this phenomenon holds in higher dimensions $\bar d=2k$ to $\bar d = 2k+1$ $(k\in {\mathbb Z}_{\geq 2})$.
\end{remark}

Using Propositions~\ref{deltaRK} and~\ref{P2N}, as well as Equations~\nn{deltaRX} and~\nn{DXT}, we arrive at an explicit formula for the obstruction density for surfaces.
\begin{proposition}\label{B2}
The   obstruction density for~$\db = 2$ is given by\begin{equation*}
\B_2 = -\frac13\,  L^{ab}\, \IIo_{ab}\, .
\end{equation*}
\end{proposition}
\noindent
Comparing the above with Proposition~\ref{gradient2}
verifies again that the obstruction density and the gradient of the Willmore energy~${\mathcal I}_2$ agree.

\subsection{Four-manifold  obstruction density}

Our goal is now to find hypersurface formula for the~$\db = 3$ obstruction density and thus prove Proposition~\ref{Calzone}. We begin by calculating~${\sf P}_3N^B$.
\begin{proposition}\label{P3N}
When~$\db = 3$,
$$
{\sf P}_3 N^B=4 \, B^\prime X^B +\frac43\, \big[U^{-1}\big]{}^{\!B}{}_{\!C} \bar D^C K-4\big( \delta_R K\big ) N^B\, ,
$$
where
\begin{equation}\label{B'}
\begin{split}
B'=&  
L^{ab}\big(2\IIo^2_{(ab)\circ}\!\!-\Whn_{ab}\big)
 -\IIo^{ab}\Bach_{ab}
 +\frac12 K^2
 %\\[1mm]&
 -3W(\hat n,\IIo^2,\hat n)+\Whn_{ab}\Whn^{ab}
 +\frac12 \Whn^\top_{abc} \Whn^{abc}\, .
 \end{split}
 \end{equation}
\end{proposition}

\begin{proof}
Using Equation~\nn{nablaN}, the first two terms in~${\sf P}_3N^C$ appearing in Equation~\nn{P23} can be written explicitly as
$$
 \left( \begin{array}{c} 8K \\[1mm] -8\IIo^{ab}\nablab_a\IIo_{bc} - 4\IIo_c^b\nablab.\IIo_b 
 +8n_c\big(\tr \IIo^3+HK\big)
 \\[2mm]
 4\IIo ^{ab}\nablab_a\nablab.\IIo_b  +
4\nablab.\IIo_a\nablab.\IIo^a
+8\tr(\IIo^2\Rho^\top) \end{array} \right) \, .
$$
To compute the remaining two terms of~${\sf P}_3N^C$, we make use of the tractor curvature and connection formul\ae\ in  Equations~\nn{curvature} and~\nn{trconn}, respectively, as well as Equation~\nn{nablaN}, and find
\begin{equation*}
\begin{split}
n^a\nabla^b \circ {\mathcal R}_{ab}^{\bm \sharp} N^C \, \, &\=
 \begin{pmatrix}
 0 \\ -\Wn_{abc}\IIo^{ab} \\[1mm]  - \nablab^b \Cn_b- W(n,P^\top,n)
 \end{pmatrix} \, , \\[2mm]
 n^a{\mathcal R}_{ab}^{\bm \sharp} \circ {\nabla}^b{}^\top N^C &\=\, \, 
 \begin{pmatrix}
 0 \\[1mm] -\Wn_{abc}\IIo^{ab}\\[1mm] -C(n,\IIo)
 \end{pmatrix}\, .
\end{split}
\end{equation*}
In the first equation we have made use of the four-manifold identity~$\nabla^bW_{bacd} = C_{cda}$. Orchestrating, we have
\begin{equation*}
{\sf P}_3N^C\! = \!
\begin{pmatrix}
8K \\[1mm] -8\IIo^{ab}\nablab_a\IIo_{bc} - 4\IIo_c^b\nablab.\IIo_b - 8W(n,\IIo,c)^\top
 +8n_c\big(\tr \IIo^3-W(n,\IIo,n)+HK \big) \\[1mm] 
 4\IIo ^{ab}\nablab_a\nablab.\IIo_b  \!+
4\nablab.\IIo_a\nablab.\IIo^a
\!+\!8\tr(\IIo^2\Rho^\top) \!-\! 4C(n,\IIo) \!-\!4\nablab^b\Cn_b\! - 4W(n,\Rho^\top\!,n)
\end{pmatrix}
 .
\end{equation*}
This expression can be  simplified by comparing it with the boundary Thomas D-operator acting on the rigidity density,
\begin{equation}\label{duck}
\frac 43\big[U^{-1}\big]{}^{\!B}{}_{\!C} \bar D^C \! K =
\begin{pmatrix}
1 & 0 & 0\\[1mm] n_b H & \delta^c_b & 0 \\[1mm] -\frac12H^2 &\! \!\!-n^c H & 1
\end{pmatrix}
\begin{pmatrix}
8K \\[1mm] -4\nablab_c K \\[1mm] -\frac 43(\bar \Delta\! - 2\bar \J)K
\end{pmatrix}
= \begin{pmatrix} 
8 K \\[1mm] -4\nablab_b K + 8n_b HK \\[1mm] -\frac43(\bar\Delta\! - 2\bar \J)K \!- \!4H^2K
\end{pmatrix} \, .
\end{equation}
 Here we have used the isomorphism~\nn{Tisomorphism}
 between sections of $\cT\Sigma$ and $\ct M^\top$.
To match this expression with terms appearing in~${\sf P}_3N^C$ as displayed above, we make use of the trace-free Mainardi Equation~\nn{tfcm} to obtain the  identity
 \begin{equation*}
 \begin{split}
 \frac12 \nablab_c K&=\IIo^{ab}\nablab_c\IIo_{ab}=\IIo^{ab}\nablab_a\IIo_{bc} + \frac12\IIo_c^b\nablab.\IIo_b^{\phantom{c}} +\Wn_{abc}^\top\IIo^{ab}\, . \end{split}
 \end{equation*}
Using this identity, Equation~\nn{nemo}, the Fialkow Equation~\nn{Fialkow} as well as Proposition~\ref{deltaRK} we have
\begin{equation*}
{\sf P}_3N^B = \frac43\big[U^{-1}\big]{}^{\!B}{}_{\!C} \bar D^C K -4N^B\delta_{\rm R} K + 4B^\prime X^B\, ,
\end{equation*}
with~$B^\prime$ given by
\begin{equation*}
\begin{split}
B^{\prime}&=\frac 23 \nablab^a\IIo^{bc}\nablab_a \IIo_{bc}\!+
(\nablab.\IIo_a)(\nablab.\IIo^a)
+\frac43\IIo^{ab}\bar\Delta\IIo_{ab}
\!-\frac43\bar \J K
 - \frac83\IIo^{ab}\nablab^c\Whn_{abc}^\top
-3W(n,\IIo^2\!,n)\\[1mm]&
+ \frac12 K^2
+\Wn_{ab}\Wn^{ab}
+2\IIo^{ab}\nablab^c\Wn_{bac}^\top
-\nablab^b \Cn_b
-W(n,\bar P,n)
-C(n,\IIo)-HW(n,\IIo,n)
\, .
\end{split}
\end{equation*}
It remains only to show that this last expression reduces to that stated in the proposition. This follows directly from Proposition~\ref{LIIo2o} and the identities
$$
\Cn_b\,\,  =\nablab^c \Wn_{cb}
 +\Wn_{acb}^\top\IIo^{ac}
 \Rightarrow
 \nablab^b \Cn_b\,\,  =\nablab^a\nablab^b \Wn_{ab}
 +\IIo^{ab}\nablab^c\Wn_{abc}^\top
 +\frac12\Wn_{abc}^\top
 \Wn_{abc}
 \, ;
$$
these  are direct consequences of Equation~\nn{divW} and~\nn{tfcm}. Noting that $n=\hat n$ along~$\Sigma$ completes the proof.
\end{proof}

We now turn to the second term\ in the holographic formula~\nn{holo-form}. We first need the following   technical result:
\begin{lemma}\label{DT}
Let $T\in \Gamma(\ct^\phi M[-2])$ and $\db = 3$. Then
$$
I\cdot D^2( X^CT) \stackrel\Sigma = -2 D^C T\, .
$$
\end{lemma}
\begin{proof}
The key to the proof is to note that for~$d = 4$, the tractor~$X^CT$ has ambient Yamabe weight~$w=-1$, so that
$$
I \cdot D(X^BT) = -\sigma \square_Y( X^B T)\, .
$$
Acting with the remaining~$I \cdot D \stackrel\Sigma= -2\delta_R$ (at this weight) and calculating explicitly using the tractor connection~\nn{trconn} gives the identity stated.
\end{proof}

Our result for two canonical normal derivatives of the rigidity density multiplied by the canonical tractor is given by the following: 

\begin{proposition} \label{DK}
Let~$\db = 3$. Then,
\begin{equation}\label{DBK}
I\cdot D^2( X^CK) \stackrel\Sigma= 4X^B B^{\prime\prime}
-\frac43\, \big[U^{-1}\big]{}^{\!B}{}_{\!C} \bar D^C K+4N^B\delta_R K \, ,
\end{equation}
where
\begin{equation}\label{B''}
\begin{split}
B^{\prime\prime}&=L^{ab}\, \IIo^2_{(ab)\circ}
 + \frac12 K^2
-4W(\hat n,\IIo^2,\hat n)
+\Whn_{ab}\Whn^{ab}
+\IIo^{ab}\IIo^{cd}W_{cabd}
+\frac12 \Whn_{abc}^\top \Whn^{abc}\, .
\end{split}
\end{equation}
\end{proposition}

\begin{proof}
Using Lemma~\ref{DT} we must compare
$$
-2D^C K =
\begin{pmatrix}
-8K \\[1mm]
4\nabla_c K\\
2(\Delta-2\J)K
\end{pmatrix}
$$
 with Equation~\nn{duck}. This yields Equation~\nn{DBK} with
\begin{equation}\label{B''1}
B^{\prime\prime} = \frac12 \nabla_n^2 K + 2 H \delta_R K +\frac 16 \bar \Delta K
-\frac32 H^2 K -\frac13 \bar \J K -\frac14 K^2 -\Rho(n,n)K\, .
\end{equation}
It remains to show that this expression reduces to Equation~\nn{B''}. For that, we need to compute two normal derivatives of the canonically extended rigidity density~$K$. Returning to Equation~\nn{nablanIIo}, we see that the following normal derivatives of curvature terms appear in $\frac12 \nabla_n^2 K$:
\begin{align*}
\nabla_n \Rn_{ab} &+2\nabla_n \Rho_{ab}\\[1mm]
&= (\nablan n^d)R_{dabe}n^e + n^cn^d\nabla_c (R_{dabe}n^e)+2\nabla_n \Rho_{ab} \\[1mm]
&\= H \Rn_{ab} + (g^{cd} - \bar g^{cd})\nabla_c(R_{dabe}n^e)
+2\nabla_n \Rho_{ab} \\[1mm]
&= H\Rn_{ab} + \nabla^d (R_{dabc}n^c) + {\nabla}^\top_{d} \Rn_{b}{}^d{}_a+2\nabla_n \Rho_{ab} \\[1mm]
&= H\Rn_{ab} + (\nabla^d n^c) R_{dabc} - n^c(\nabla_b {{R_{dac}}^d} + \nabla_c {{R_{da}}^d}_b)  + {\nabla}^\top_d\Rn_b{}^d{}_a +2\nabla_n \Rho_{ab} \\[1mm]
&\= H\Rn_{ab} + (\IIo^{dc} + Hg^{dc})R_{dabc} + n^c\nabla_b R_{ac} -g_{ab}\nablan \J + {\nabla}^\top_d\Rn_b{}^d{}_a\ .
\end{align*}
In the fourth line above we  used the Bianchi identity for the Riemann tensor. Upon contraction with~$\IIo^{ab}$ the above yields
\begin{equation*}
\begin{split}
\IIo^{ab}\big(\nablan\Rn_{ab} + 2\nablan\Rho_{ab}\big)=&
\, \, W(n,\IIo^2,n)
-HW(n,\IIo,n)-3H\tr(\IIo\Rho)
+\IIo^{ab}\IIo^{cd}W_{cabd}\\[1mm]
&-\tr(\IIo^2\Rho)
+ \IIo^{ab}\nablab_a \! \stackrel{\sss n}\Rho{\!}_b^{\!\top}
+K\Rho(n,n)
 - \IIo^{ab}{\nablab}^d\Wn_{bad}^\top\, .
\end{split}
\end{equation*}
Combining the above expression and Equation~\nn{nablanIIo}, we have\begin{equation}\label{nablan2IIo}
\begin{split}
\IIo^{ab}\nablan^2\IIo_{ab} \stackrel\Sigma=&\phantom{+} \IIo^{ab}\big(\nablan\Rn_{ab} + 2\nablan\Rho_{ab} \big) - 2\IIo^{ab}\IIo_{cb}\nablan\IIo^c_a 
+ K\nablan\rho 
\\[1mm]&
+ \rho\IIo^{ab}\nablan\IIo_{ab} -3\rho\tr(\IIo\Rho^\top) + 2\tr(\IIo^2\Rho^{\!\top}) - \IIo^{ab}\nabla_a\nabla_b\rho
\\[1mm]
=&\phantom{+}
\frac13\IIo^{ab}\bar\Delta\IIo_{ab}
-\frac13\bar \J K
+2H\tr \IIo^3
+\frac12H^2K
+\frac14K^2
-2W(n,\IIo^2,n)
\\[1mm]&
-2HW(n,\IIo,n)
+\IIo^{ab}\IIo^{cd}W_{cabd}
+K\Rho(n,n)
 - \frac43\IIo^{ab}\nablab^c\Wn_{bac}^\top
\, .
\end{split}
\end{equation}
Here we have used Equations~\nn{nablann}, \nn{nablanIIo}
and~\nn{Fialkow} as well as the $3\times 3$ matrix identity~$\tr\IIo^4 = \frac12 K^2$. In addition, to handle the term $\IIo^{ab}\nabla_a\nabla_b \rho$, the following consequence of the Mainardi Equation~\nn{Main} is needed:
$$
\IIo^{ab}\nablab_a\nablab_b H = \frac13 \IIo^{ab}\bar\Delta\IIo_{ab} - \IIo^{ab}\nablab_a\!\!\stackrel{\sss n}{\Rho}{\!\!}_b^\top - \frac13 \bar \J K - \tr(\IIo^2\bar\Rho) - \frac13 \IIo^{ab}\nablab^c\Wn_{bac}^\top \, .
$$
The computation of~$\nablan^2 K$ also requires the cross term~$(\nablan\IIo_{ab})(\nablan\IIo^{ab})$. Again using Equation~\nn{nablanIIo} we have
$$
\nablan\IIo_{ab} \stackrel\Sigma= \Wn_{ab} -n_an_b K + n_{(a}\nablab.\IIo_{b)} - H\IIo_{ab} - \IIo^2_{ab} + \frac12 g_{ab} K\, .
$$
Contracting this identity on itself we find
\begin{equation*}
\begin{split}
(\nablan\IIo_{ab})(\nablan\IIo^{ab}) &= \Wn_{ab}\Wn^{ab} - 2HW(n,\IIo,n) -2W(n,\IIo^2,n) +\frac12(\nablab.\IIo_b)(\nablab.\IIo^b)
\\[1mm]& + H^2K + 2H\tr(\IIo^3) + \frac12K^2 \, .
\end{split}
\end{equation*}
Orchestrating this and Equation~\nn{nablan2IIo}, we have
\begin{equation*}
\begin{split}
\frac12\nablan^2 K&=
\frac12 (\nablab.\IIo_a)(\nablab.\IIo^a)
+\frac13\IIo^{ab}\bar\Delta\IIo_{ab}
+ \frac34 K^2-4W(n,\IIo^2,n)
+\Wn_{ab}\Wn^{ab}
+\frac32H^2K\\[1mm]&-4HW(n,\IIo,n)
+4H\tr\IIo^3
-\frac13\bar \J K
+\IIo^{ab}\IIo^{cd}W_{cabd}
+\Rho(n,n)K
 -\frac43 \IIo^{ab}\nablab^c\Wn_{bac}^\top
\, .
\end{split}
\end{equation*}
Substituting this result into Equation~\nn{B''1} and simplifying  gives
\begin{equation*}
\begin{split}
B^{\prime\prime}&=\frac 13 \nablab^a\IIo^{bc}\nablab_a \IIo_{bc}+
\frac12(\nablab.\IIo_a)(\nablab.\IIo^a)
+\frac23\IIo^{ab}\bar\Delta\IIo_{ab}
-\frac23\bar \J K
 - \frac43\IIo^{ab}\nablab^c\Wn_{bac}^\top\\[1mm]&
-4W(n,\IIo^2,n)
+ \frac12 K^2
+\Wn_{ab}\Wn^{ab}
+\IIo^{ab}\IIo^{cd}W_{cabd}
\, .
\end{split}
\end{equation*}
Using Proposition~\ref{LIIo2o} gives the result~\nn{B''}, which completes the proof.
\end{proof}

We now
 complete the proof of Proposition~\ref{Calzone}
 giving the obstruction density for four-manifolds:
 
% we have the   obstruction density for four-manifold hypersurfaces:
%
%\begin{proposition}\label{Calzone}
%The 
%obstruction density for~$\db = 3$ is given by
%$$
%\B_3 \!= \!\frac16\Big[ \!L^{ab\,}\! \!
%%\circ 
%\hspace{.2mm}
%\big(3\IIo^2_{\!(ab)\!\hspace{.2mm}\circ} \hspace{.4mm}\!\!-\! \hspace{.2mm}\Whn_{ab}\big) -
%%\tr(\IIo B) 
%\IIo^{ab}\!B_{ab}
% + K^2\! -7\Whn^{ab}\IIo^2_{ab} + 2\Whn_{ab}\Whn^{ab}\! + \IIo^{ab}\IIo^{cd}W_{\!cabd} + \Whn\!_{abc}^\top\Whn^{abc}\Big] .
%$$
%\end{proposition}
\begin{proof}[Proof of Proposition~\ref{Calzone}]
Using Equation~\nn{holo-form} and Propositions~\ref{P3N} and~\ref{DK}, the obstruction density is given by
$$
\B_3 = \frac{1}{72}\bar D^A \big[4X_A(B^{\prime} + B^{\prime\prime})\big] \, .
$$
Substituting the formul\ae\,~\nn{B'} and \nn{B''} for~$B^\prime$ and~$B^{\prime\prime}$ into the above and using Lemma~\ref{DXT} gives the result.
\end{proof}

\subsection{Examples and umbilicity}

Our results can be easily applied to explicit metrics
and hypersurfaces. As we shall explain below, this gives both an independent check of our computations
and also generates interesting examples. 

For surfaces, total umbilicity suffices to ensure vanishing of the obstruction density (this follows by inspection of the formula for $\B_2$ given in Proposition~\ref{B2}).
The following proposition shows this is not true in dimension $\bar d=3$.

\begin{proposition}
Total umbilicity ({\it i.e.}, vanishing of $\IIo$ everywhere) is not a sufficient condition for the obstruction density to vanish.
\end{proposition}

\begin{proof}
 See the
 explicit counterexample on the third line of the table below. \end{proof}
 
 \begin{remark}
 There is no in principle difficulty constructing examples of umbilic hypersurfaces with non-vanishing obstruction density in dimensions ~$\bar d>3$. Also, albeit in Lorentzian signature, the Schwarzschild example below provides another $\bar d=3$ example of this.
 \end{remark}

In the following table we give
some sample four manifold metrics, hypersurfaces and their obstruction density~${\mathcal B}_3$. Additionally, $\bar \sigma$ denotes the conformal unit density in the scale determined by the quoted metric, $H$ the mean curvature, $\IIo$ the trace-free second fundamental form and $L$ the membrane
rigidity density. Non-vanishing expressions that are too long to be displayed are denoted  by a~$\star$.
These results were generated with the aid of a computer software package~\cite{grtensor}.
\begin{center}
\begin{tabular}{|c|c|c|c|c|c|c|}
\hline
{}$ds^{2^{\phantom{2^2}\!\!\!\!}}$&$\Sigma$&$\bar \sigma$&$H$&$\IIo$&${\mathcal B}_3$&$L$\\[.2mm]
\hline\hline
${\sss dx^2+(1+\alpha x)(dy^2+dz^2+dw^2)}$&
${\sss x=0}$&\!${\sss x+\frac\alpha4x^2-\frac{\alpha^2}8x^3+\frac{5\alpha^3}{64}x^4}$\!&
$\frac\alpha2$&
0&
0&
0\\[1mm]
${\sss dx^2+(1+\alpha x)dy^2+(1-\alpha x)dz^2 +dw^2}$&
${\sss x=0}$&
${\sss x-\frac{\alpha^2}8 x^3}$&
${ 0}$&
${\sss \frac\alpha2(dy^2-dz^2)}$&
${\sss \frac{\alpha^4}{12}}$&
${ 0}$\\[1mm]
%Example3.mpl
$\!{\sss dx^2+(1+z)dy^2+(1+w)dz^2+(1+y)dw^2}$\!&
${\sss x=0}$&
$\star$&
$0$&
$0$&
$\star$&
$0$\\[1mm]
${\sss -(1-\frac{\sss 2M}{r})dt^2
+\frac{dr^2}{1-\frac{\sss 2M}r}+r^2 d\Omega^2 }$&
\!${\sss r=R+2M}$\!&
$\star$&
\!\scalebox{.6}{Eq.~\nn{SII}}
\!\!\!&
\scalebox{.6}{Eq.~\nn{SII}}
&
\!\scalebox{.6}{Eq.~\nn{SB}}
\!\!\!&
\!\scalebox{.6}{Eq.~\nn{SL}}
\!\!\!
\\[2mm]
 \hline
\end{tabular}
\end{center}
\medskip
The first line of the table gives a metric and an umbilic but  non-minimal hypersurface for which the obstruction density vanishes. In the second line, the hypersurface is minimal but non-umbilic. The third line is particularly interesting, the hypersurface is both umbilic and minimal, yet  the obstruction density does not vanish.
Finally in the fourth line, we use the fact that all local formul\ae\  presented in this paper can be extended to Lorentzian structures (provided the conormal is nowhere null), and study a spherical shell of radius $R$ in the Schwarzschild metric. 
The conormal ${\bm d}r$ is spacelike for $R>0$, timelike for $R<0$ and lightlike at the horizon $R=0$, so that strictly the formul\ae\  presented below apply only to the spacelike case.
In that case the  
 second fundamental form
 is given by
\begin{equation}\label{SII}
H=\frac13
\frac{M+2R}{\sqrt{R(R+2M)^3}}\, ,\quad
\IIo=
-\frac13(M-R)\sqrt{\frac{R+2M}{R}}\, 
\Big(\frac{2Rdt^2}{(R+2M)^3}+d\Omega^2\Big)\, ,
\end{equation}
so that the hypersurface~$\Sigma$ is neither minimal nor umbilic, save for the distinguished radius $R=M$ where umbilicity holds. 
At that value the membrane rigidity density also vanishes; at  general $R$ it is given by
\begin{equation}\label{SL}
L=\frac{2(M-R)(M^2+6MR+R^2)}{R\sqrt{R(R+2M)^9}}\, .
\end{equation}
The obstruction density is
\begin{equation}\label{SB}
{\mathcal B}_3=\frac{2M^4+4M^3 R+ 30 M^2 R^2 -8MR^3-R^4
 }{27R^2(R+2M)^6}
\end{equation}
and has one simple real root in the region  $R>-2M$ at $R/M\approx 2.9$.

Finally let us explain how the fornul\ae\ for the obstruction density can be generated without simply 
computing each term in Proposition~\ref{Calzone}.
This is based on the two iterative recursions given in~\cite[Lemma 2.4 and Proposition 4.9]{GW15}. Firstly, given a four manifold with metric $g$ and hypersurface~$\Sigma$ labeled by a defining function $s$, we improve the defining function to 
$$s_0= \frac{s}{|\nabla s|}$$
which ensures that $|\nabla s_0|\big|_\Sigma=1$. It is highly propitious to work with a unit defining function~$s_1$ with $|\nabla  s_1|=1$. For applications, it suffices to require $|\nabla  s_1|=1+s^\ell f$ for smooth $f$ and $\ell$ of sufficiently high order. This is achieved using the recursion of~\cite[Lemma 2.4]{GW15}. For the first three examples tabulated above, $s_1=x$ is already a unit defining function, but for the Schwarzschild case,
we performed this algorithm  explicitly using the computer software package~\cite{grtensor}.

Supposing now that $s$ is a unit defining function, we next improve this to a conformal unit defining density evaluated in the scale determined by the metric $g$. This can be achieved using the recursion of~\cite[Proposition 4.9]{GW15}. 
In practice this simply amounts to repeatedly  computing the leading failure of $I_{\sigma}^2$ to equal unity and then 
adding a next to leading correction to $\sigma$ to correct for this. In Appendix~\ref{udfif} we present  the first four improvement terms for  a unit defining function~$s$  as well as the corresponding obstruction densities for hypersurface  dimensions $\bar d=1,2,3$ and~$4$. These were computed using the symbolic manipulation software~\cite{form}.
We have verified their correctness by checking that
the resulting~$I_{\sigma}^2$ equals unity to the required order for explicit metric and hypersurface examples, including those given above.

We note that the result for $B_4$ quoted in the appendix, when evaluated along~$\Sigma$ gives the obstruction density for 
 five-manifold hypersurfaces expressed in terms of a unit defining function. This is a novel result that is simple to use in practical applications. 
 There is no in principle difficulty employing the unit defining function identities of Sections~\ref{dfvc} and~\ref{cfm} and their higher order analogs to evaluate~${\mathcal B}_5$ in terms of hypersurface invariants. 

\section*{Acknowledgements}
A.G. and A.W. would like to thank R. Bonezzi, C.R. Graham and Y. Vyatkin for helpful comments and discussions. The authors gratefully acknowledge support from the Royal Society of New Zealand via Marsden Grant 13-UOA-018 and the UCMEXUS-CONACYT grant CN-12-564. A.W. thanks for their warm hospitality the University of Auckland and the Harvard University Center for the Fundamental Laws of Nature. A.W. was also supported by a Simons Foundation Collaboration Grant for Mathematicians ID 317562.

\appendix

\section{Conformal unit defining function formul\ae\  and five-manifold obstruction density}\label{udfif}

\allowdisplaybreaks

Let $s$ be a defining function for a hypersurface~$\Sigma$ in a $d$-dimensional Riemannian manifold~$(M,g)$. The following formul\ae\  improve this to $s_k$ which determines a defining density $\sigma_k$ in the scale~$g$ such that
$$
I_{\sigma_{k}}^2 = 1 + \sigma^k f\, , \quad k<d\, ,
$$  
for some smooth $f\in\Gamma(\ce M[-k])$:

\begin{eqnarray*}
s_0\!\!&=&\!\!s\\[3mm]
s_1\!\!&=&\!\!s_0+\frac{1}{2(d-1)}\, s^2\, \nabla.n\\[3mm]
s_2\!\!&=&\!\!s_1+\frac1{6(d-2)(d-1)}\, s^3\, \big(2(\nabla.n)^2-(d-4)\, \nabla_n\nabla.n+2(d-1)\, \J\big)\\[3mm]
s_3\!\!&=&\!\!s_2
+\frac1{24(d-3)(d-2)(d-1)}\, s^4\, 
\Big(\frac{5d-6}{d-1}\, (\nabla.n)^3
-4(2d-9)\, \nabla.n \, \nabla_n\nabla.n\\[1mm]
&+&\!\!\!\!(d-4)(d-6)\, \nabla_n^{\, 2}\nabla.n        
+3(d-2)\, \Delta\nabla.n
+4(2d-3)\, \nabla.n\, \J
-2(d-6)(d-1)\, \nabla_n \J\Big) \\[3mm]
s_4\!\!&=&\!\!s_3+\frac{1}{120(d-4)(d-3)(d-2)(d-1)}\, s^5\, \Big(
\frac{13d^3-62d^2+100d-48}{(d-2)(d-1)^2}\, (\nabla.n)^4\\[1mm]
&+&\!\!\frac{398d^2\!-880d\!-49d^3\!+576}{(d-2)(d-1)}\, (\nabla.n)^2\, \nabla_n\!\nabla.n
+\frac{13d^3\!-156d^2\!+524d\!-384}{d-1}\, \nabla.n \, \nabla_n^2\nabla.n\\[1mm]
&+&\!\!\frac{25d^2-94d+72}{d-1}\, \nabla.n\  \Delta\nabla.n
+\frac{11d^4-149d^3+668d^2-1112d+576}{(d-2)(d-1)}\, (\nabla_n\nabla.n)^2\\[1mm]
&-&\!\!(d-4)(d-6)(d-8)\, \nabla_n^3\nabla.n
-\frac{(3d^2-22d+16)(d-3)}{d-1}\, (\nabla \nabla.n)^2\\[1mm]
&-&\!\!4(d-4)(d-3)\, \Delta\nabla_n\nabla.n\!
-3(d\!-2)(d\!-8)\, \nabla_n \Delta \nabla.n
+\frac{2(13d^2\!-58d\!+72)}{(d-2)}\, \!(\nabla.n)^2 \J\\[1mm]
&-&\!\!\frac{4(4d^3-35d^2+88d-72)}{d-2}(\nabla_n\nabla.n)\, \J
+\frac{4(d-4)(d-3)(d-1)}{d-2}\, \J^2\\[1mm]
&-&\!\!6(3d-4)(d-6)\nabla.n\, \nabla_n\J
+8(d-3)(d-1)\, \Delta\J	+2(d-8)(d-6)(d-1)\, \nabla_n^2\J
\Big)\, .
\end{eqnarray*}
When $k=d$, we have $I^2=1+\sigma^d B_{\bar d}$ where the obstruction density ${\mathcal B}_{\bar d}=B_{\bar d}|_\Sigma$. Thus the following determine the obstruction density in dimensions $d=2,3,4$ and~$5$. 
\begin{eqnarray*}
%B_1&=&-2\, \nabla.n\\[2mm]
B_1&=&-(\nabla.n)^2-\nabla_n\nabla.n-\J\\[2mm]
B_2&=&-\frac1{12}\, \big(2\Delta\nabla.n+2\, \nabla_n^2\nabla.n+8\, \nabla.n\, \nabla_n\nabla.n+3\, (\nabla.n)^3+8\nabla.n\, \J + 8\, \nabla_n\J\big)\\[1mm]
B_3&=&-\frac{1}{108}\, \Big(9\, \nabla_n\Delta\nabla.n+12\, \nabla.n\, \Delta\nabla.n+6\, \nabla.n\,\nabla_n^2\nabla.n +3\, (\nabla\nabla.n)^2\\[1mm]
&&\qquad\:\:+\ 6(\nabla_n\nabla.n)^2
+18\, (\nabla.n)^2\, \nabla_n\nabla.n+4\, (\nabla.n)^4\\[1mm]&&\qquad\:\:+\ 9\, \Delta\J+18\, \nabla_n^2 \J+36\, \nabla.n\, \nabla_n\J+18\, \nabla_n\nabla.n\, \J+18\, (\nabla.n)^2\, \J
     \Big)\\[2mm]
B_4&=&
%\begin{verbatim}
  \frac1{46080}\Big(
          -\,  312 \nabla.n (\nablan\nabla.n)^2
          + 352 \nabla.n (\nablan^3\nabla.n)
          - 155 \nabla.n^2 \nabla.n^3
          \\[1mm]
&&\qquad\
          + 44\,  \nabla.n^2 (\nablan^2\nabla.n)
          - 832 \nabla.n^3 (\nablan\nabla.n)
          + 480 (\nablan\nabla.n) (\nablan^2\nabla.n)
        \\[1mm]
&&\qquad\  
          +\,  96\,  (\nablan^4\nabla.n)
          + 32 \Delta\nablan^2\nabla.n
          + 288 \nabla.n\Delta\nablan\nabla.n 
          \\[1mm]
&&\qquad\
          +\,  256\,  \nablan \Delta \nablan \nabla.n
          + 96 (\nabla_a\nabla.n)\nabla^a\nablan \nabla.n
          - 552\nabla.n (\nabla_a\nabla.n)\nabla^a\nabla.n 
          \\[1mm]
&&\qquad\
          -\,  608\,  (\nabla_a\nabla.n) \nablan\nabla^a\nabla.n
          - 1436  \nabla.n^2\Delta\nabla.n
          - 1248  (\nablan\nabla.n)\Delta\nabla.n
          \\[1mm]
&&\qquad\
          -\,  2176\,   \nabla.n\nablan\Delta\nabla.n
          - 864 \nablan^2\Delta\nabla.n
          - 288 \Delta^2\nabla.n
          - 2304\,   \nabla.n (\nablan\nabla.n)\J
                 \\[1mm]
&&\qquad\
          -\,  832\,   \nabla.n^3\J
          - 640  (\nablan^2\nabla.n)\J
          - 384\,   (\Delta\nabla.n)\J
          + 1024 \J^2 \nabla.n
          + 512 \J \nablan\J
               \\[1mm]
&&\qquad\
          -\,  2496\,   \nabla.n^2 \nablan\J
          - 2304 (\nablan\nabla.n)\nablan\J 
          + 1536  \J\nablan\J
          - 2432\,   \nabla.n\nablan^2\J
                    \\[1mm]
&&\qquad\
          -\, 768\,  \nablan^3\J
          - 256 \Delta\nablan\J
          - 512\,  (\nabla_a\nabla.n)\nabla^a\J
          - 2176 \nabla.n \Delta\J 
          - 2048 \nablan\Delta\J\Big)\, .
%\end{verbatim}
\end{eqnarray*}
The above yields ${\mathcal B}_1=B_1|_\Sigma=0$
and it is not difficult to show that~$B_2|_\Sigma$
agrees with~${\mathcal B}_2$ as in  Proposition~\ref{B2}. The expressions above for $B_3$ and $B_4$  provide an efficient way to compute the obstruction densities ${\mathcal B}_3$ and ${\mathcal B}_4$ for explicit metrics, but expressing the above formul\ae\  in terms of standard hypersurface invariants is rather tedious. For $(M,g)$ flat and~$\bar d=3$, this computation was performed in~\cite{GW15}
and agrees with our generally curved  expression for~${\mathcal B}_3$ given in Proposition~\ref{Calzone}.

\bibliographystyle{hep}
\bibliography{Master}

\end{document}